\documentclass[12pt,notitlepage]{amsart}
\usepackage{latexsym,amsfonts,amssymb,amsmath,amsthm}
\usepackage{color}
\usepackage[inner=1.0in,outer=1.0in,bottom=1.0in, top=1.0in]{geometry}

\newcommand{\mz}{\ensuremath{\mathbb Z}}
\newcommand{\mr}{\ensuremath{\mathbb R}}
\newcommand{\mh}{\ensuremath{\mathbb H}}

\newcommand{\mc}{\ensuremath{\mathbb C}}

\newcommand{\shortmod}{\ensuremath{\negthickspace \negthickspace \negthickspace \pmod}}

\newcommand{\half}{\ensuremath{ \frac{1}{2}}}

\newcommand{\intR}{\int_{-\infty}^{\infty}}

\newcommand{\thalf}{\tfrac12}
\newcommand{\sumstar}{\sideset{}{^*}\sum}
\newcommand{\leg}[2]{\left(\frac{#1}{#2}\right)}
\newcommand{\e}[2]{e\left(\frac{#1}{#2}\right)}
\newcommand{\SL}[2]{SL_{#1}(\mathbb{#2})}
\newcommand{\GL}[2]{GL_{#1}(\mathbb{#2})}
\newcommand{\leftexp}[2]{{\vphantom{#2}}^{#1}{#2}}

\theoremstyle{plain}		
	\newtheorem{mytheo}{Theorem}[section]
	
	\newtheorem{myprop}[mytheo]{Proposition}
	\newtheorem{mycoro}[mytheo]{Corollary}
     \newtheorem{mylemma}[mytheo]{Lemma}

\theoremstyle{remark}

\numberwithin{equation}{section}
\begin{document}
\author{Xiaoqing Li}
\address{Deptartment of Mathematics \\
State University of New York at Buffalo \\
Buffalo, NY, 14260
}
\email{XL29@buffalo.edu}

\author{Matthew P. Young} 
\address{Department of Mathematics \\
	  Texas A\&M University \\
	  College Station \\
	  TX 77843-3368}
\email{myoung@math.tamu.edu}
\thanks{This material is based upon work supported by the National Science Foundation under agreement Nos. DMS-0901035 (X.L.), DMS-0758235 (M.Y.),  and DMS-0635607 (X.L. and M.Y.).  Any opinions, findings and conclusions or recommendations expressed in this material are those of the authors and do not necessarily reflect the views of the National Science Foundation.}

% \begin{abstract}
% \end{abstract}
\today
\title{The $L^2$ restriction norm of a $GL_{3}$ Maass form}
\maketitle
\section{Introduction and main result}
Kac's question, ``Can one hear the shape of a drum?'' \cite{Kac} is a famous example of the interest in the connections between geometrical data and spectral information, which continues to be a fascinating direction of study.
Weyl's law gives a beautiful asymptotic formula for the counting function of the eigenvalues on a compact Riemannian manifold in terms of geometrical quantities (dimension, volume, etc.).  

In quantum chaos, a key issue is the behavior of the eigenfunctions as the eigenvalue becomes large.  In particular, one would like to know if the eigenfunctions behave like random waves, or if they can concentrate on certain subdomains.
The influentical QUE conjecture of Rudnick and Sarnak \cite{RS1} asserts that the quantum measures associates to the eigenstates tend (in the weak-$^*$ sense) to the volume measure provided that the manifold has negative curvature.

We are naturally led to studying the sizes of Laplace eigenfunctions which can be measured in various ways.  For instance, one may consider the $L^p$ norms for $p \geq 2$.  Alternatively, one may consider $L^p$ norms of the eigenfunction restricted to some subset of its domain.  In the arithmetical setting one has a commuting family of Hecke operators in addition to the Laplacian, and so it is natural to consider the behavior of these Maass forms.
There are a small handful of results in this direction for $GL_2$ automorphic forms.  In particular, \cite{IS} \cite{Xia} \cite{BH} \cite{Templier} \cite{Milicevic} studied the supremum norm in different aspects. Sarnak and Watson \cite{SarnakBAMS} have announced a proof of a sharp bound (up to $\lambda^{\varepsilon}$) on the $L^4$ norm of Maass forms in the spectral aspect.

Reznikov \cite{Reznikov} wrote an influential preprint studying $L^2$ restriction problems of automorphic forms restricted to certain curves.  Since then, there have appeared a number of papers studying very general problems of bounding the $L^p$ norm of the restriction of the eigenfunction of the Laplacian to a submanifold of a Riemannian manifold, including \cite{BGT} (see also \cite{Hu}) with some very general results which are sharp in their generality, and \cite{BR} which in particular stresses the problem of finding lower bounds.  However, in the context of automorphic forms these general results are not sharp and it is desirable to prove stronger results and to understand what the true order of magnitude should be, whether it can be proven or not.  
Sarnak nicely explains some of the issues in studying such restriction problems, especially the connection with the Lindel\"{o}f hypothesis on pages 5 and 6 in \cite{SarnakReznikov}  (see also \cite{SarnakMorowetz}).  

In a slightly different direction, Michel and Venkatesh \cite{MV} proved a ``subconvex'' geodesic restriction theorem (see their Section 1.4) for the geodesic Fourier coefficients of $GL_2$ automorphic forms.

In this paper, we study
a novel restriction problem for a $GL_3$ Maass form restricted to a codimension $2$ submanifold (essentially $GL_2 \times \mr^{+}$).  Such a restricted function has nice invariance properties; it is invariant by $\SL{2}{Z}$ on the left and by $O_2(\mr)$ on the right, and it is natural to understand how it fits into the $GL_2$ picture.  For instance, one can ask what is the inner product of this restricted function with a given $\SL{2}{Z}$ Maass form (or more generally, we ask for the spectral decomposition).
In fact, the Rankin-Selberg $L$-function for $GL_3 \times GL_2$ is constructed along these lines.  There are many examples of such period integrals giving values of $L$-functions, in particular we mention \cite{GrossPrasad}.

Our main result is the following.
\begin{mytheo}
\label{thm:normbound}
 Let $F$ be a Hecke-Maass form of type $(\nu_1, \nu_2)$ for $\SL{3}{Z}$ that is in the tempered spectrum of $\Delta$ (meaning $\text{Re}(\nu_1) = \text{Re}(\nu_2) = 1/3$ or alternately the Langlands parameters $i\alpha, i\beta, i\gamma$ defined by \eqref{eq:alpha}-\eqref{eq:gamma} are purely imaginary), with Laplace eigenvalue $\lambda_F(\Delta) = 1 + \half(\alpha^2 + \beta^2 + \gamma^2)$, and with $L^2$ norm equal to $1$.  Then we have
\begin{equation}
\label{eq:N(F)}
 N(F):= \int_0^{\infty} \int_{\SL{2}{\mz} \backslash \mathcal{H}^2} \Big|F\begin{pmatrix} z_2 y_1 & \\ & 1\end{pmatrix}\Big|^2 \frac{dx_2 dy_2}{y_2^2} \frac{dy_1}{y_1} \ll_{\varepsilon} \lambda_F(\Delta)^{\varepsilon} |A_F(1,1)|^2,
\end{equation}
where
\begin{equation}
\label{eq:z2}
 z_2 = \begin{pmatrix} 1 & x_2 \\ & 1 \end{pmatrix} \begin{pmatrix} y_2 &  \\ & 1 \end{pmatrix} y_2^{-\half},
\end{equation}
$A_F(1,1)$ is the first Fourier coefficient of $F$, 
and the implied constant depends only on $\varepsilon > 0$.
\end{mytheo}
Remarks.  
This is the first sharp codimension $2$ restriction result, as well as the first such result in a higher rank ($GL_3$) context.  

It is a pleasant exercise to compute the analog of $N(F)$ when $F$ is a Maass form for $\SL{2}{Z}$, that is $N(F) := \int_0^{\infty} |F(\begin{smallmatrix} y &  \\ & 1 \end{smallmatrix})|^2 \frac{dy}{y}$: one obtains the second moment along the critical line of the completed $L$-function associated to $F$.  

For context, the bound of \cite{BGT} would give $N(F) \ll \lambda_F(\Delta)^{1/2} (\log{\lambda_F(\Delta)})^{1/2} |A_F(1,1)|^2$.  Strictly speaking, their bound does not apply since $\SL{3}{\mz} \backslash \SL{3}{\mr}/SO_{3}(\mr)$ is not compact, but more importantly our bound is much stronger and is probably sharp (up to the $\varepsilon$).

The problem of bounding $N(F)$ was given in \cite{SarnakReznikov}, where he remarks that the Lindel\"{o}f hypothesis gives the bound stated in our Theorem \ref{thm:normbound}.  In our work, this will be apparent in Section \ref{section:scaling}.

S.D. Miller \cite{M} showed that ``almost all'' cusp forms are tempered, and the Archimedean Ramanujan-Selberg conjecture implies that all cusp forms are tempered.
With the Langlands functoriality conjectures, one can show $|A_F(1,1)|^2 \ll \lambda_F(\Delta)^{\varepsilon}$, but this is difficult to establish unconditionally as it is related to showing the non-existence of a Landau-Siegel zero for the Rankin-Selberg $L$-function $L(s, F \times \overline{F})$ (see Proposition \ref{prop:Landau} below for the precise relation).  Fortunately, for Maass forms $F$ that arise as a symmetric-square lift of a $\SL{2}{Z}$ Maass form (equivalent, $F$ is self-dual), Ramakrishnan and Wang \cite{RW} have shown that $|A_F(1,1)| \ll \lambda_F(\Delta)^{\varepsilon}$, and hence we have the following
\begin{mycoro}
 Let notation be as in Theorem \ref{thm:normbound}.  If $F$ is self-dual then
\begin{equation}
 N(F) \ll_{\varepsilon} \lambda_F(\Delta)^{\varepsilon}.
\end{equation}
\end{mycoro}

We end the introduction with a brief outline of the rest of the paper.  Sections \ref{section:background} and \ref{section:RS} are devoted to standard material on automorphic forms and Rankin-Selberg $L$-functions.  By the spectral theory for $GL_2$, Parseval's formula, and Plancherel's formula, we derive a pleasant formula connecting the $L^2$ norm of the restriction to $GL_2 \times \mr$ of the $GL_3$ Maass form to families of the $GL_3 \times GL_2$ $L$-functions (Theorem \ref{thm:normformula}).  By applying Stirling's formula to the Archimedean factors of the $L$-functions, we break the families into pieces at appropriate scales; this is done in Section \ref{section:scaling}.  Section \ref{section:bilinearforms} provides some standard tools in harmonic analysis as well as some variations on Gallagher's large sieve inequalities.  We are left with establishing suitable bounds for averages of Rankin-Selberg $L$-functions in different ranges.  In many ranges (but not all), the desired bounds correspond to a Lindel\"{o}f bound on average, while in all ranges, dropping all but one term recovers the convexity bound.  By applying the approximate functional equations for the Rankin-Selberg $L$-functions, we are led to prove Theorem \ref{thm:mainthm}: a mean value estimate for the $L$-functions.  An application of the $GL_2$ Kuznetsov formula transforms the spectral sums into mean values with standard exponential sums, giving \eqref{eq:MABdef}.  In \eqref{eq:MABdef}, when $A$ is small, i.e., $B$ is large, a straightforward application of Gallagher's large sieve (Lemma \ref{lemma:largesieve}) gives the desired bound; this is carried out in the rest of Section \ref{section:largesieve}.  When $A$ is big, i.e., $B$ is small, we need to use the $GL_3$ Voronoi formula to shorten the sum (see Section \ref{section:Voronoi}) before applying the large sieve; this last step is done in Section \ref{section:largesieve2}.  This basic outline is similar to \cite{Y}, but virtually all the details are changed.  The essential difference is that here the $GL_3$ form is varying, while in \cite{Y}, the $GL_3$ Maass form is fixed.  Here we found a simple method to take care of the uniformity in our estimates (see Lemma \ref{lemma:Phiproperties}).  Stationary phase arguments in \cite{Y} are avoided here and instead we only need to use integration by parts.

\section{Acknowledgements}
We thank Enrico Bombieri and Peter Sarnak for inviting us for the special year at the Institute for Advanced Study
which provided us a nice environment to work. Especially we thank Peter Sarnak for introducing us to this nice problem and for his
encouragement. We also thank John Friedlander, Zeev Rudnick, Soundararajan, and Akshay Venkatesh for their interest, comments, and corrections.

\section{Background on automorphic forms and $L$-functions}
\label{section:background}
We rely on \cite{Goldfeld} for many of the basic facts of higher rank automorphic forms.

Let $m = (m_1, m_2) \in \mathbb{Z}^2$, and let $\nu = (\nu_1, \nu_2) \in \mc^2$.  The Jacquet-Whittaker function for $\SL{3}{Z}$ takes the form
\begin{equation}
 W_J(z, \nu, \psi_m) = \int_{\mr^3} I_{\nu}(w_3 u z) \overline{\psi_m(u)} du_{12} du_{13} du_{23},
\end{equation}
where
\begin{equation}
 w_3 = \begin{pmatrix} & & 1 \\ & -1 & \\ 1 & & \end{pmatrix}, \qquad u = \begin{pmatrix} 1 & u_{12} & u_{13} \\ & 1 & u_{23} \\  & & 1 \end{pmatrix},
\end{equation}
and
\begin{equation}
 \psi_m(u) = e(m_1 u_{23} + m_2 u_{12}), \qquad I_{\nu}(z) = y_1^{\nu_1 + 2 \nu_2} y_2^{2 \nu_1 + \nu_2},
\end{equation}
for
\begin{equation}
 z = 
\begin{pmatrix} 
1 & x_{12} & x_{13} \\
 & 1 & x_{23} \\
& & 1
\end{pmatrix}
\begin{pmatrix} 
y_1 y_2 &  &  \\
 & y_1 &  \\
& & 1
\end{pmatrix}
\in \mathcal{H}^3 := \GL{3}{\mr}/O_3(\mr) \mr^{\times}.
\end{equation}
In many situations it is more convenient to work with the Langlands parameters defined by
\begin{align}
\label{eq:alpha}
 &i\alpha = -\nu_1 - 2 \nu_2 + 1, \\
 &i\beta = -\nu_1 + \nu_2, \\
 &i\gamma = 2 \nu_1 + \nu_2 - 1.
\label{eq:gamma}
\end{align}
Suppose $F$ is a Maass form of type $\nu = (\nu_1, \nu_2)$ for $\SL{3}{Z}$.  
The temperedness of $F$ means that $\alpha, \beta, \gamma$ defined above by \eqref{eq:alpha}-\eqref{eq:gamma} are real.
Thanks to Jacquet, Piatetski-Shapiro, and Shalika, we have the following Fourier-Whittaker expansion (see (6.2.1) of \cite{Goldfeld})
\begin{equation}
\label{eq:Fourierexpansion}
 F(z) = \sum_{\gamma \in U_2(\mz) \backslash \SL{2}{Z}} \sum_{m_1 \geq 1} \sum_{m_2 \neq 0} \frac{A_F(m_1, m_2)}{m_1 |m_2|} W_J\Big( %\begin{pmatrix} m_1 m_2 & & \\ & m_1 & \\ & & 1\end{pmatrix} 
M \begin{pmatrix} \gamma & \\ & 1 \end{pmatrix} z, \nu, \psi_{1, \frac{m_2}{|m_2|}} \Big),
\end{equation}
where $U_2(\mz)$ is the group of $2 \times 2$ integer, upper trianguler matrices with ones along the diagonal, and $M$ is the diagonal matrix with entries $m_1 |m_2|, m_1, 1$ along the diagonal.
In later sections we may use the shorthand $W_J(z)$ to denote $W_J(z, (\nu_1, \nu_2), \psi_{1,1})$.  The dual form associated to $F$ (see Proposition 6.3.1 of \cite{Goldfeld}) is of type $(\nu_2, \nu_1)$ and has a similar Fourier expansion but with $A(m_2, m_1)$ as its $(m_1, m_2)$th Fourier coefficient.  
If furthermore $F$ is an eigenform for the full Hecke ring, then $A(m_2, m_1) = \overline{A(m_1, m_2)}$ (see \cite{Goldfeld}, p.271).
Note that switching $\nu_1$ and $\nu_2$ replaces the Langlands parameters $(i\alpha, i\beta, i\gamma)$ by $(-i\gamma, -i\beta, -i\alpha)$.

In our work we crucially require the $GL_3$ Voronoi formula first proved by Miller and Schmid \cite{MS} (see \cite{GolfeldLi} for another proof), which we now state.
Suppose $k=0$ or $1$, and $\psi(x)$ is a smooth, compactly-supported function on the positive reals.  Define
\begin{equation}
 \widetilde{\psi}(s) = \int_0^{\infty} \psi(x) x^{s} \frac{dx}{x}.
\end{equation}
For $\sigma > -1 + \max\{-\text{Re}(i\alpha), -\text{Re}(i\beta),  -\text{Re}(i\gamma)\}$, define
\begin{equation}
 \psi_k(x) = \frac{1}{2 \pi i} \int_{(\sigma)} (\pi^3 x)^{-s} \frac{\Gamma\left(\frac{1+s+i\alpha+k}{2}\right)\Gamma\left(\frac{1+s+i\beta+k}{2}\right)\Gamma\left(\frac{1+s+i\gamma+k}{2}\right)}{\Gamma\left(\frac{-s-i\alpha+k}{2}\right)\Gamma\left(\frac{-s-i\beta+k}{2}\right)\Gamma\left(\frac{-s-i\gamma+k}{2}\right)} \widetilde{\psi}(-s) ds.
\end{equation}
Then define
\begin{gather}
 \Psi_{+}(x) = \frac{1}{2 \pi^{3/2}} ( \psi_0(x) + \frac{1}{i} \psi_1(x) ) \\
\Psi_{-}(x) = \frac{1}{2 \pi^{3/2}} (\psi_0(x) - \frac{1}{i} \psi_1(x)).
\end{gather}
\begin{mytheo}[\cite{MS}]
\label{thm:Voronoi}
 Let $\psi(x)$ be smooth and compactly-supported on the positive reals.  Suppose $d, \overline{d}, c \in \mz$ with $c \neq 0$, $(c,d) = 1$, and $d \overline{d} \equiv 1 \pmod{c}$.  Then
\begin{multline}
 \sum_{n > 0} A_F(m,n) \e{n \overline{d}}{c} \psi(n) = c \sum_{n_1 | cm} \sum_{n_2 > 0} \frac{A_F(n_2, n_1)}{n_1 n_2} S(md, n_2; mc/n_1) \Psi_{+} \leg{n_2 n_1^2}{c^3 m}
\\
+ c \sum_{n_1 | cm} \sum_{n_2 > 0} \frac{A_F(n_2, n_1)}{n_1 n_2} S(md, -n_2; mc/n_1) \Psi_{-} \leg{n_2 n_1^2}{c^3 m},
\end{multline}
where $S(a,b;c)$ is the usual Kloosterman sum.
\end{mytheo}

Now we recall the spectral theory of automorphic forms for $\SL{2}{Z}$.  Let $u_j(z)$ be an orthonormal basis of Hecke-Maass cusp form for $\SL{2}{Z}$ (as in \cite{IwaniecSpectral}, p.117).  Write the Fourier expansion as (see (3.4) and (1.26) of \cite{IwaniecSpectral})
\begin{equation}
\label{eq:ujFourier}
 u_j(z) = \sum_{n \neq 0} \rho_j(n) W_{\half + it_j}(n z), \quad z = \begin{pmatrix} 1 & x \\ & 1 \end{pmatrix} \begin{pmatrix} y &  \\ & 1 \end{pmatrix} y^{-\half},
\end{equation}
where
\begin{equation}
 W_{\half+i\tau}(z) = 2 \sqrt{|y|} K_{i\tau} (2 \pi |y|) e(x)
\end{equation}
and $K_{i \tau}$ is the usual $K$-Bessel function.  Let $\lambda_j(n)$ be the $n$-th Hecke eigenvalue of $u_j$, whence
\begin{equation}
\label{eq:rholambda}
 \rho_j(\pm n) = \rho_j(\pm 1) \lambda_j(n) |n|^{-\half}.
\end{equation}
The scaling is such that $\lambda_j(1) = 1$ and the Ramanujan conjecture implies $|\lambda_j(p)| \leq 2$ for $p$ prime.  By \cite{IwaniecSmall}, \cite{HL}, we have
\begin{equation}
\label{eq:rhoj1}
 t_j^{-\varepsilon} \ll \alpha_j:=\frac{|\rho_j(1)|^2}{\cosh(\pi t_j)} \ll t_j^{\varepsilon},
\end{equation}
which establishes the scaling of $|\rho_j(1)|^2$ in terms of $t_j$.  In this work we do not require the bounds \eqref{eq:rhoj1}, but we mention them since it is illuminating to understand the scaling, and to contrast the behavior with $\SL{3}{Z}$ Maass forms.  We return to this discussion in Section \ref{section:RS}. 

Next we discuss the continuous spectrum.  The $\SL{2}{Z}$ Eisenstein series is defined by
\begin{equation}
 E(z_2,s) = \half \sum_{c,d \in \mz, (c,d) = 1} \frac{y_2^s}{|cz_2 + d|^{2s}} = \half \sum_{\gamma \in U_2(\mathbb{Z}) \backslash \SL{2}{Z}} \text{Im}(\gamma z_2)^s,
\end{equation}
and has the Fourier expansion (see Theorem 3.4 of \cite{IwaniecSpectral})
\begin{equation}
\label{eq:EisensteinFourier}
 E(z_2,s) = y_2^s + \rho(s) y_2^{1-s} + \sum_{n \neq 0} \rho(n,s) W_s(nz_2),
\end{equation}
where 
\begin{equation}
 \rho(s) = \sqrt{\pi} \frac{\Gamma(s-\half)}{\Gamma(s)} \frac{\zeta(2s-1)}{\zeta(2s)}, \qquad 
% \end{equation}
% \begin{equation}
 \rho(n,s) = \pi^s \frac{\lambda(n,s)}{\Gamma(s) \zeta(2s) |n|^{\half}},
\end{equation}
and
\begin{equation}
 \lambda(n,s) = \sum_{ad = |n|} \leg{a}{d}^{s-\half}.
\end{equation}
Observe $\overline{\lambda}(n, \half + i\tau) =\lambda(n, \half + i\tau)$ for real $\tau$.  The analog of \eqref{eq:rhoj1} for the Eisenstein series is essentially a classical fact about the Riemann zeta function (see \cite{T}, (3.5.1) and (3.6.5)) that
$\tau^{-\varepsilon} \ll \zeta(1 + 2i \tau) \ll \tau^{\varepsilon}$, giving
\begin{equation}
\label{eq:rhoj1E}
 \tau^{-\varepsilon} \ll \alpha_{\tau}:=\frac{|\rho(1, \half + i \tau)|^2}{\cosh(\pi \tau)} = \frac{1}{|\zeta(1+2i \tau)|^2} \ll \tau^{\varepsilon}.
\end{equation}

We recall the well-known spectral theorem; see \cite{IwaniecSpectral} for example.
\begin{mytheo}
Suppose $f \in L^2(\SL{2}{Z} \backslash \mh)$.  Then
\begin{equation}
\label{eq:GL2spectral}
 f(z) = \sum_{j \geq 0} \langle f, u_j \rangle u_j(z) + \frac{1}{4 \pi} \intR \langle f, E(\cdot, \half + it) \rangle E(z, \half + it) dt.
\end{equation}
If $f$ and $\Delta f$ are smooth and bounded then the sum converges absolutely and uniformly on compact sets.  
Furthermore, the Parseval formula says
\begin{equation}
\label{eq:Parseval}
 ||f||^2 = \sum_{j \geq 0} |\langle f, u_j \rangle|^2 + \frac{1}{4 \pi} \intR |\langle f, E(\cdot, \half + it) \rangle|^2 dt.
\end{equation}
\end{mytheo}
We also recall the Kuznetsov formula (Theorem 9.3 of \cite{IwaniecSpectral}).
\begin{mytheo}[Kuznetsov]
\label{thm:Kuznetsov}
Let notation $\alpha_j, \lambda_j(n), \alpha_{\tau}, \lambda(n, \half + i\tau)$ be defined as above, and suppose $h(r)$ satisfies
\begin{equation}
 \begin{cases}
  h(r) = h(-r), \\
 h \text{ is holomorphic in } |Im(r)| \leq \half + \delta, \\
 h(r) \ll (1 + |r|)^{-2-\delta},
 \end{cases}
\end{equation}
for some $\delta > 0$.  Then we have
\begin{multline}
 \sum_{j \geq 1} \alpha_j \lambda_j(m) \lambda_j(n) h(t_j) + \frac{1}{ 4 \pi} \intR \alpha_{\tau} \lambda(m, \thalf + i\tau) \overline{\lambda}(n, \thalf + i\tau) h(\tau) d\tau
=
\\
\delta_{m,n} \pi^{-2}  \intR r \tanh(\pi r) h(r) dr 
+
\sum_{c=1}^{\infty} \frac{S(m,n;c)}{c} H\leg{4 \pi \sqrt{mn}}{c}, 
\end{multline}
where
\begin{equation}
\label{eq:Hdef}
 H(x) = \frac{2i}{\pi} \intR r h(r) \frac{J_{2ir}(x)}{\cosh(\pi r)} dr = \frac{2i}{\pi} \int_0^{\infty} r h(r) \frac{J_{2ir}(x) - J_{-2ir}(x)}{\cosh(\pi r)} dr.
\end{equation}
\end{mytheo}

%\begin{equation}
 %\frac{\lambda_f(n)}{n^s}$
%\end{equation}
% is an $L$-function as defined in Section 5 of \cite{IK}, say.  Let $q(f,s)$ denote the analytic conductor of $L(f,s)$.  We desire an upper bound for $|L(f,s)|^2$ valid inside the critical strip that depends only on the rough magnitude of $q(f,s)$, so that we may use a common formula for a class of $f$ and $s$.  Essentially, we wish to separate variables.  This is carried out with the following
\begin{mylemma}[Approximate functional equation]
\label{lemma:AFE}
 Let $L(s,f) = \sum_{n \geq 1} \lambda_f(n) n^{-s}$ be an $L$-function as in Chapter 5 of \cite{IK} such that the completed $L$-function is entire.  Let $q(f,s)$ denote the analytic conductor of $L(f,s)$ as defined on p.95 of \cite{IK}, and suppose that $q(f,s) \leq Q$ for some number $Q > 0$.  Then there exists a function $W(x)$ depending on $Q$ and $\varepsilon$ only, such that $W$ is supported on $x \leq Q^{\half + \varepsilon}$ and satisfying
\begin{equation}
\label{eq:WAFE}
 x^j W^{(j)}(x) \ll_{j,\varepsilon} 1,
\end{equation}
where the implied constant depends on $j$ and $\varepsilon$ only (not $Q$), 
such that
\begin{equation}
 |L(\thalf + it,f)|^2 \ll Q^{\varepsilon} \int_{-\log{Q}}^{\log{Q}} |\sum_{n \geq 1} \frac{\lambda_f(n)}{n^{\half + it + iv}} W(n) |^2 dv %+ \int_{-Q^{\varepsilon}}^{Q^{\varepsilon}} |\sum_n \frac{\lambda_f(n)}{n^{\half + it + iv}} W(n) |^2 dv 
+  O(Q^{-100}),
\end{equation}
where the implied constant depends on $\varepsilon$, $W$, and the degree of $L(f,s)$ only.
\end{mylemma}
Remark.  The power of positivity makes this formulation extremely simple;  an exact formula for $|L(\thalf + it, f)|^2$ would be much more complicated.  The point is that $W$ does not vary with $f$ and $t$ as long as $q(f, \half + it) \leq Q$.
\begin{proof}
 The usual approximate functional equation (Theorem 5.3 of \cite{IK}) gives
\begin{equation}
\label{eq:usualAFE}
 L(\thalf + it,f ) = \sum_{n \geq 1} \frac{\lambda_f(n)}{n^{\half + it}} V_{f,t}(n/\sqrt{q}) + \epsilon_{f,t} \sum_{n \geq 1} \frac{\overline{\lambda_f(n)}}{n^{\half - it}} V_{f,-t}^{*}(n/\sqrt{q}),
\end{equation}
where $q$ is the archimedean part of the conductor (see \cite{IK}, p.94), 
\begin{equation}
\label{eq:Vdef}
 V_{f,t}(x) = \frac{1}{2 \pi i} \int_{(2)} x^{-u} %\frac{\Gamma(\frac{\kappa_1 + \half + it + u}{2})}{\Gamma(\frac{\kappa_1 + \half + it}{2})} \dots \frac{\Gamma(\frac{\kappa_d + \half + it + u}{2})}{\Gamma(\frac{\kappa_d + \half + it}{2})} 
\frac{\gamma(f,\half + it + u)}{\gamma(f, \half + it)}
e^{u^2} \frac{du}{u},
\end{equation}
in which
\begin{equation}
 \gamma(f,s) = \pi^{-ds/2} \prod_{j=1}^{d} \Gamma\big(\frac{s+\kappa_j}{2}\big),
\end{equation}
$V_{f,t}^*$ is given by 
\begin{equation}
 V_{f,t}^*(x) = \frac{1}{2 \pi i} \int_{(2)} x^{-u} %\frac{\Gamma(\frac{\kappa_1 + \half + it + u}{2})}{\Gamma(\frac{\kappa_1 + \half + it}{2})} \dots \frac{\Gamma(\frac{\kappa_d + \half + it + u}{2})}{\Gamma(\frac{\kappa_d + \half + it}{2})} 
\frac{\gamma(\overline{f},\half + it + u)}{\gamma(\overline{f}, \half + it)}
e^{u^2} \frac{du}{u},
\end{equation}
and $\varepsilon_{f,t}$ is a complex number with absolute value $1$.  Note that there is a misprint on p.94 \cite{IK}, since $\gamma(f,s) \neq \gamma(\overline{f}, s)$ in general; the correct statement is that if $f$ has parameters $\{ \kappa_1, \dots, \kappa_d \}$, then its dual has parameters $\{\overline{\kappa_1}, \dots, \overline{\kappa_d} \}$ (see (2.8) of \cite{RS} for example).

Shifting the contour to the right and using Stirling's approximation shows that $V_{f,t}(x/\sqrt{q})$ is very small for $x \geq q(f, \thalf+ it)^{\half + \varepsilon}$.  We choose a $W_0$ satisfying \eqref{eq:WAFE} such that multiplication by $W_0(n)$ in \eqref{eq:usualAFE} introduces an error of size $O(Q^{-100})$ to the value of $L(\half + it, f)$; for example, one can take $W_0$ to be identically $1$ for $n \leq Q^{\half + \varepsilon}$ and then have it smoothly decay to be zero for $n \geq 2 Q^{\half + \varepsilon}$.  Having inserted this weight into the $n$-sums, we then apply the integral representation definition of $V_{f,t}(n/\sqrt{q})$ (shifted to the point $\sigma>0$ to be chosen later) and reverse the orders of summation and integration.  Using Cauchy's inequality, we obtain
\begin{multline}
 |L(\thalf + it, f)|^2 \leq \frac{2}{2\pi} \Big| \int_{(\sigma)} %(2\pi)^{??} \frac{\Gamma(\frac{\kappa_1 + \half + it + u}{2})}{\Gamma(\frac{\kappa_1 + \half + it}{2})} \dots \frac{\Gamma(\frac{\kappa_d + \half + it + u}{2})}{\Gamma(\frac{\kappa_d + \half + it}{2})} 
q^{u/2}
\frac{\gamma(f,\half + it + u)}{\gamma(f, \half + it)}
\frac{e^{u^2}}{u}
\sum_{n \geq 1} \frac{\lambda_f(n) W_0(n)}{n^{\half + it + u}}  
   du \Big|^2
\\
+ (\text{similar}) + O(Q^{-100}),
\end{multline}
where the ``similar'' term has $\lambda_f(n)$ replaced by $\overline{\lambda_f(n)}$, and $t$ replaced by $-t$.  The integrand decays very rapidly as a function of $\text{Im}(u)$ due to the exponential decay of $e^{u^2}$, so that we can truncate the $u$-integrals at $|\text{Im}(u)| \leq \log{Q}$ without making a new error term.  By Stirling's formula we have (see p.100 of \cite{IK})
\begin{equation}
q^{u/2} \frac{\gamma(f,\half + it + u)}{\gamma(f, \half + it)} \ll Q^{\text{Re}(u)/2} \exp(\frac{\pi d}{2} |u|).
\end{equation}
Thus
\begin{equation}
 |L(\thalf + it, f)|^2 \ll_{\sigma} Q^{\sigma} \int_{u=\sigma + iv, |v| \leq \log{Q}}  \Big|
\sum_{n \geq 1} \frac{\lambda_f(n) W_0(n)}{n^{\half + it + u}}  
  \Big|^2 dv + (\text{similar}) + O(Q^{-100}).
\end{equation}
Letting $W(n) = W_0(n) n^{-\sigma}$, taking $\sigma = \varepsilon$, and noting that the ``similar'' term is actually identical to the displayed term (it is the complex conjugate), we finish the proof.
\end{proof}

\section{Rankin-Selberg $L$-functions}
\label{section:RS}
In this work we require knowledge of various types of Rankin-Selberg $L$-functions.  In particular, we need the explicit integral representation, functional equation, and the connection with the $L^2$ norm.

It is instructive to first recall the well-known case of $GL_2 \times GL_2$.  For this, we have the following integral representation, if $u_j$ is even or odd:
\begin{equation}
\label{eq:ujujRS}
\zeta(2s) \int_{\SL{2}{Z} \backslash \mathbb{H}} |u_j(z)|^2 E(z,s) \frac{dx dy}{y^2} = 2^{-1} \pi^{-s} \frac{\Gamma(\frac{s}{2})^2}{\Gamma(s)} |\rho_j(1)|^2 \Gamma(\tfrac{s}{2} - it_j) \Gamma(\tfrac{s}{2} + it_j) L(s, u_j \times \overline{u_j}),
\end{equation}
where
\begin{equation}
 L(s, u_j \times \overline{u_j}) = \zeta(2s)\sum_{n=1}^{\infty} \frac{|\lambda_j(n)|^2}{n^s}.
\end{equation}
This is derived by the unfolding method and from explicit knowledge of the Mellin transform of the product of two $K$-Bessel functions.  In this way, we deduce a functional equation for $L(s, u_j \times \overline{u_j})$ from that of the Eisenstein series.

On the other hand, the Fourier expansion \eqref{eq:EisensteinFourier} shows that the Eisenstein series has a simple pole at $s=1$ with residue $\frac{\pi}{2 \zeta(2)}$.  Thus taking the residues of both sides of \eqref{eq:ujujRS}, we have that 
\begin{equation}
\label{eq:GL2L2RS}
1 = \langle u_j, u_j \rangle =  \frac{|\rho_j(1)|^2}{\cosh(\pi t_j)} \text{Res}_{s=1} L(s, u_j \times \overline{u_j}).
\end{equation}
Thus upper/lower bounds on the residue of the $L$-function correspond to lower/upper (respectively) bounds on $|\rho_j(1)|^2$.  It is well-known that $L(s, u_j \times \overline{u_j}) = \zeta(s) L(s, \text{Sym}^2 u_j)$, where $L(s, \text{Sym}^2 u_j) = \zeta(2s) \sum_n \lambda_j(n^2) n^{-s}$ is the Gelbart-Jacquet lift \cite{GJ} of $u_j$, which is known to correspond to a self-dual $\SL{3}{Z}$ Maass form.  Then estimates for the $L$-functions translate to estimates on \eqref{eq:rhoj1}.

It is less classical to understand the behavior of the first Fourier coefficient of $F$, a $\SL{3}{Z}$ Maass form.  For this, we have
\begin{myprop}
\label{prop:Landau}
 Let $F$ be a $\SL{3}{Z}$ Hecke-Maass form that is in the tempered spectrum of $\Delta$.  Then for some absolute constant $c > 0$, we have
\begin{equation}
\label{eq:GL3L2RS}
 \langle F, F \rangle = c  |A_F(1,1)|^2 \text{Res}_{s=1} L(s, F \times \overline{F}),
\end{equation}
where we write $A_F(m,n) = A_F(1,1) \lambda_F(m,n)$, and
\begin{equation}
 L(s, F \times \overline{F}) = \zeta(3s) \sum_{m \geq 1} \sum_{n \geq 1} \frac{|\lambda_F(m,n)|^2}{(m^2 n)^s}.
\end{equation}
\end{myprop}
In contrast to \eqref{eq:GL2L2RS}, \eqref{eq:GL3L2RS} does not exhibit an external scaling factor analogous to $\frac{1}{\cosh(\pi t_j)}$ (an artifact of the definition of the Whittaker functions), so that assuming standard conjectures on the size of the residue at $s=1$ of $L(s, F \times \overline{F})$, we have $\lambda_F(\Delta)^{-\varepsilon} \ll |A_F(1,1)| \ll \lambda_F(\Delta)^{\varepsilon}$, 

The proof follows the same lines as \eqref{eq:ujujRS} but requires a much more difficult Archimedean integral involving the product of two $GL_{3}$ Whittaker functions.  This crucial integral was computed by Stade \cite{StadeGL3GL3}.

If $F$ is not tempered then the formula is not so clean and instead $c$ depends loosely on the form in the sense that $1 \ll c \ll 1$ with absolute implied constants.  By the way, a similar phenomenon already occurs in the $\SL{2}{Z}$ case if $t_j$ is not real.  We only assume the form is tempered at the end of the proof.
\begin{proof}
In Section 7.4 of \cite{Goldfeld}, it is shown that
\begin{equation}
\label{eq:RSformulaGL3}
 \zeta(3s) \langle F \overline{G}, E(\cdot, \overline{s}) \rangle = A_F(1,1) \overline{A_G(1,1)} L(s, F \times \overline{G}) G_{\nu, \nu'}(s),
\end{equation}
where $F$ and $G$ are $\SL{3}{Z}$ Hecke-Maass forms of types $\nu$ and $\nu'$, respectively, 
\begin{equation}
E(z, s) = \half \sum_{\gamma \in \widehat{\Gamma} \backslash \SL{3}{Z}} \det(\gamma z)^s,
\end{equation}
$\widehat{\Gamma}$ is the subset of elements of $\SL{3}{Z}$ with lower row $(0,0,1)$, 
\begin{equation}
 L(s, F \times \overline{G}) = \zeta(3s) \sum_{m \geq 1} \sum_{n \geq 1} \frac{\lambda_F(m,n) \overline{\lambda_G(m,n)}}{(m^2 n)^s},
\end{equation}
and
\begin{equation}
 G_{\nu,\nu'}(s) = \int_0^{\infty} \int_0^{\infty} W_J(y, \nu, \psi_{1,1}) \overline{W_J}(y, \nu', \psi_{1,1}) (y_1^2 y_2)^s \frac{dy_1 dy_2}{y_1^3 y_2^3}.
\end{equation}
Here we wrote $A(m,n) = A(1,1) \lambda(m,n)$ so that the scaling on $\lambda$ is such that $\lambda(1,1) = 1$ and the Ramanujan conjecture implies $|\lambda(1,p)| \leq 3$.  Stade \cite{StadeGL3GL3} computed this integral, but it is a little tricky to convert notation between \cite{StadeGL3GL3} and \cite{Goldfeld}.

First we state Stade's formula (\cite{StadeGL3GL3}, (1.2)), observing that Stade's $y_1$ and $y_2$ are switched compared to ours, and that his $a_1$ is our $i\gamma$, and his $a_2$ is our $i\beta$:
\begin{multline}
 \pi^{-3s/2} \Gamma(\frac{3s}{2}) \int_0^{\infty} \int_0^{\infty} W_J^S(y, \nu, \psi_{1,1}) W_J^S(y, \nu', \psi_{1,1}) (y_1^2 y_2)^s \frac{dy_1 dy_2}{y_1^3 y_2^3} 
\\
= \pi^{-9s/2} \prod_{j=1}^{3} \prod_{j'=1}^{3} \Gamma\big(\frac{s + i\alpha_j + i \alpha_{j'}'}{2}\big),
\end{multline}
where we write the Langlands parameters as $(\alpha,  \beta,  \gamma) = (\alpha_1, \alpha_2, \alpha_3)$, and $W_J^S$ denotes Stade's normalization of the Whittaker function (defined in \eqref{eq:WJS} below).  We need to convert between $W_J^S$ and $W_J$.  We can determine the normalization of Stade's Whittaker function from the integral representation \cite{StadeGL3GL3}, (1.1):
\begin{equation}
\label{eq:WJS}
 W_J^{S}(y, \nu, \psi_{1,1}) = 2^3 y_1^{1 + \frac{i \beta}{2}} y_2^{1 - \frac{i \beta}{2}} \int_0^{\infty} K_{\mu}(2 \pi y_1 \sqrt{1 + u}) K_{\mu}(2 \pi y_2 \sqrt{1 + u^{-1}}) u^{3i\beta/4} \frac{du}{u},
\end{equation}
where $\mu = \half (i\gamma - i\alpha)$.  Changing variables $u = v^2$, and comparing to (6.1.3) of \cite{Goldfeld}, we see that
\begin{equation}
 W_J^{S}(y, (\nu_1, \nu_2), \psi_{1,1}) = c W_J^*(y, (\nu_2, \nu_1), \psi_{1,1}),
\end{equation}
where $c=4$, and $W_J^*$ is defined on p.154 of \cite{Goldfeld} as
\begin{equation}
\label{eq:WJ*}
 W_J^*(z, \nu, \psi_{1,1}) = \pi^{\half - 3 \nu_1 - 3\nu_2} \Gamma \big(\frac{3\nu_1}{2}\big) \Gamma \big(\frac{3\nu_2}{2}\big) \Gamma \big(\frac{3\nu_1 + 3\nu_2 -1}{2}\big) W_J(z, \nu, \psi_{1,1}).
\end{equation}
Notice that $3 \nu_1 = 1-i(\beta-\gamma)$, $3\nu_2 = 1 - i(\alpha - \beta)$,  $3\nu_1 + 3\nu_2 - 1 = 1-i(\alpha-\gamma)$.

It is clear from \eqref{eq:WJS} that $\overline{W_J^S}(y, \nu, \psi_{1,1}) = W_J^S(y, \overline{\nu}, \psi_{1,1})$.
Thus we obtain
\begin{equation}
 G_{\nu, \nu}(s) = \pi^{-3s}
\frac{\prod_{j=1}^{3} \prod_{j'=1}^{3} \Gamma\big(\frac{s + i\alpha_j - i \overline{\alpha_j'}}{2} \big)}
{16 \pi^{1-3\nu_1 -3\overline{\nu_1} - 3 \nu_2 - 3 \overline{\nu_2}} |\Gamma(\frac{3 \nu_1}{2})|^2  |\Gamma(\frac{3 \nu_2}{2}) |^2 |\Gamma(\frac{3 \nu_1 + 3\nu_2 -1}{2})|^2  \Gamma(\frac{3s}{2})}.
% \frac{\Gamma(\frac{s + i(\alpha - \overline{\alpha})}{2})\Gamma(\frac{s + i(\beta - \overline{\beta})}{2})\Gamma(\frac{s + i(\gamma - \overline{\gamma})}{2}) \Gamma(\frac{s + i(\alpha-\overline{\beta})}{2}) \Gamma(\frac{s + i(\alpha-\overline{\gamma})}{2}) \Gamma(\frac{s + i(\beta-\overline{\alpha})}{2}) \Gamma(\frac{s + i(\beta-\overline{\gamma})}{2})\Gamma(\frac{s + i(\gamma-\overline{\alpha})}{2}) \Gamma(\frac{s + i(\gamma-\overline{\beta})}{2})}
% {16 \pi^{3s+1-3\nu_1 -3\overline{\nu_1} - 3 \nu_2 - 3 \overline{\nu_2}} \Gamma(\frac{3 \nu_1}{2}) \Gamma(\frac{3 \overline{\nu_1}}{2}) \Gamma(\frac{3 \nu_2}{2})\Gamma(\frac{3 \overline{\nu_2}}{2}) \Gamma(\frac{3 \nu_1 + 3\nu_2 -1}{2}) \Gamma(\frac{3 \overline{\nu_1}+ 3 \overline{\nu_2} -1}{2} ) \Gamma(\frac{3s}{2})}.
\end{equation}
Furthermore, by Proposition 7.4.4 of \cite{Goldfeld}, we have $E^*(z,s) = \pi^{-3s/2} \Gamma(\frac{3s}{2}) \zeta(3s) E(z,s)$ has a simple pole at $s=1$ with residue $2/3$.  Taking the residue at $s=1$ of both sides of \eqref{eq:RSformulaGL3}, and using the fact that $\alpha, \beta, \gamma$ are real, we obtain for some nonzero absolute constant $c$
\begin{equation}
 \langle F, F \rangle = c  |A_F(1,1)|^2 \frac{|\Gamma(\frac{1+i(\alpha-\beta)}{2})|^2 |\Gamma(\frac{1+i(\alpha-\gamma)}{2})|^2 |\Gamma(\frac{1+i(\beta-\gamma)}{2})|^2}
{|\Gamma(\frac{3 \nu_1}{2})|^2 |\Gamma(\frac{3 \nu_2}{2})|^2 |\Gamma(\frac{3 \nu_1 + 3 \nu_2 -1}{2})|^2} \text{Res}_{s=1} L(F \times \overline{F}, s).
\end{equation}
Notice that the ratio of gamma factors above precisely cancel!  This completes the proof.
\end{proof}

\begin{mylemma}[\cite{Molteni}]
\label{lemma:Molteni}
 Let $F$ be a Hecke-Maass cusp form for $\SL{3}{Z}$ %with $L^2$ norm equal to $1$.  
,let $A_F(m,n)$ be its $(m,n)$-th coefficient as in \eqref{eq:Fourierexpansion}, and suppose that the $L$-function associated to $F$ has analytic conductor $Q(F)$ defined by
\begin{equation}
\label{eq:Fconductor}
 Q(F) = (1+ |\alpha|)(1+ |\beta|)(1+|\gamma|).
\end{equation}
Then for any $\varepsilon > 0$ we have
\begin{equation}
\label{eq:RankinSelbergsumA}
 \sum_{m n \leq x} |A_F(m,n)|^2 \ll_{\varepsilon} |A_F(1,1)|^2 x^{1 + \varepsilon} Q(F)^{\varepsilon}.
\end{equation}
The implied constant is independent of $F$.
\end{mylemma}
This is actually a variation on Molteni's result \cite{Molteni}.  He proved such a bound but with the condition $m^2 n \leq x$ on the left hand side of \eqref{eq:RankinSelbergsumA}.  
\begin{proof}
 Suppose without loss of generality that $A_F(1,1) = 1$.  In the left hand side of \eqref{eq:RankinSelbergsumA}, use the Hecke relation $A_F(m,n) = \sum_{d | (m,n)} A_F(m/d, 1) A_F(1,n/d)$, apply Cauchy's inequality to the sum over $d$, and reverse the orders of summation.  Then apply Molteni's bound to the inner sum over $m$, say, followed by a second application to the sum over $n$.
\end{proof}

In this paper we work extensively with the Rankin-Selberg $L$-functions $L(s, F \times \overline{u_j})$.  The necessary Archimedean integral for this case is given by the following
\begin{myprop}[\cite{Bump}, \cite{StadeDuke}]
\label{prop:Gj}
 Let
\begin{equation}
\label{eq:Gtaudef}
 G_\tau(s) = 4 \int_0^{\infty} \int_0^{\infty} K_{i \tau}(2 \pi y_2) W_J\left(\begin{pmatrix} y_1 y_2 & &  \\ & y_1 &  \\  & & 1 \end{pmatrix}, (\nu_1, \nu_2), \psi_{1,1} \right) (y_1^2 y_2)^{s - \half} y_2^{\half} \frac{dy_2}{y_2^2} \frac{dy_1}{y_1}.
\end{equation}
Then
\begin{equation}
\label{eq:Gtau}
 G_\tau(s) = %2^4
\frac{\pi^{-3s} \Gamma\left(\frac{s-i\tau -i\alpha}{2} \right)\Gamma\left(\frac{s-i\tau -i\beta}{2} \right)\Gamma\left(\frac{s-i\tau -i\gamma}{2} \right)\Gamma\left(\frac{s+i\tau -i\alpha}{2} \right)\Gamma\left(\frac{s+i\tau -i\beta}{2} \right)\Gamma\left(\frac{s+i\tau -i\gamma}{2} \right)}
{\pi^{-\frac32 + i\alpha-i\gamma} \Gamma\left(\frac{1 + i\gamma - i\beta}{2} \right)\Gamma\left(\frac{1 + i\beta-i\alpha}{2} \right)\Gamma\left(\frac{1 + i\gamma-i\alpha}{2} \right)}.
\end{equation}
\end{myprop}
\begin{proof}
Bump \cite{Bump} proved a formula like this but without an explicit constant in front.  We shall refer to \cite{StadeDuke}.  We first remark how to translate notation.  By comparing the equation at the top of page 318 of \cite{StadeDuke} with (6.1.3) of \cite{Goldfeld}, we see that Stade's $W_{(3, \nu)}(y_2, y_1)$ is the same as $W_J^*(y, \nu, \psi_{1,1})$, where $W_J^*$ is defined by \eqref{eq:WJ*}.  Then Stade shows (see (7.8) and the equation on p.358 of \cite{StadeDuke}, though note there is a misprint in the parameter of the Bessel function) that
\begin{multline}
 \int_0^{\infty} \int_0^{\infty} W_J^*(y, \nu, \psi_{1,1}) K_{i \tau}(2 \pi y_2) (y_1^2 y_2)^s \frac{dy_1 dy_2}{y_1^2 y_2^2} 
\\
= 4^{-1} \pi^{-3s} 
\frac{\Gamma\left(\frac{s-i\tau -i\alpha}{2} \right)\Gamma\left(\frac{s-i\tau -i\beta}{2} \right)\Gamma\left(\frac{s-i\tau -i\gamma}{2} \right)\Gamma\left(\frac{s+i\tau -i\alpha}{2} \right)\Gamma\left(\frac{s+i\tau -i\beta}{2} \right)\Gamma\left(\frac{s+i\tau -i\gamma}{2} \right)}
{\pi^{-\frac32 + i\alpha-i\gamma} \Gamma\left(\frac{1 + i\gamma - i\beta}{2} \right)\Gamma\left(\frac{1 + i\beta-i\alpha}{2} \right)\Gamma\left(\frac{1 + i\gamma-i\alpha}{2} \right)}.
\end{multline}
Then using \eqref{eq:WJ*} we convert this into \eqref{eq:Gtau}, as desired.
\end{proof}

\begin{myprop}
\label{prop:GL3GL2RS} 
 Suppose $F$ is a $\SL{3}{Z}$ Hecke-Maass form as in \eqref{eq:Fourierexpansion}, and $u_j$ is a $\SL{2}{Z}$ Hecke-Maass form.  Define 
\begin{equation}
\label{eq:Lambdaintegralrep}
 \mathcal{L}(s, F \times \overline{u_j}) = \int_0^{\infty} \int_{\SL{2}{Z} \backslash \mathcal{H}^2} \overline{u_j}(z_2) F\begin{pmatrix} z_2 y_1 & \\ & 1 \end{pmatrix} y_1^{2s-1} \frac{dx_2 dy_2}{y_2^2} \frac{dy_1}{y_1},
\end{equation}
where $z_2$ is as in \eqref{eq:z2}.  If $u_j$ is even then
\begin{equation}
\label{eq:LambdaL}
 \mathcal{L}(s, F \times \overline{u_j}) = \overline{\rho_j}(1) L(s, F\times u_j) G_{t_j}(s),
\end{equation}
where $G_{\tau}(s)$ is given in Proposition \ref{prop:Gj} and where
\begin{equation}
\label{eq:GL3GL2RS}
 L(s, F \times u_j) = \sum_{m_1 \geq 1} \sum_{m_2 \geq 1} \frac{A_F(m_1, m_2) \lambda_j(m_2)}{(m_1^2 m_2)^s}.
\end{equation}
If $u_j$ is odd then \eqref{eq:Lambdaintegralrep} vanishes.
\end{myprop}
Remarks.  \begin{itemize}
           \item In \eqref{eq:GL3GL2RS} we break with our convention of defining $L$-functions only for multiplicative Dirichlet series.  We do so because it simplifies our forthcoming formulas.
		\item In \cite{GT} the $y_1$-integral is called the rank-lowering operator whose analytic properties are studied.
		\item The fact that $\mathcal{L}(s, F \times \overline{u_j})$ vanishes for $u_j$ odd indicates that \eqref{eq:Lambdaintegralrep} is not the desired integral representation for this $L$-function, but nevertheless we continue to use this notation.
          \end{itemize}
\begin{proof}
Inserting the Fourier expansion for $F$, \eqref{eq:Fourierexpansion}, and unfolding the integral, we obtain
\begin{equation}
 \mathcal{L}(s, F \times \overline{u_j}) = \int_0^{\infty} \int_0^{1} \int_0^{\infty} \overline{u_j}(z_2) \sum_{m_1 \geq 1} \sum_{m_2 \neq 0} \frac{A_F(m_1, m_2)}{m_1 |m_2|}  W_J\Big(M \begin{pmatrix} z_2 y_1 & \\ & 1 \end{pmatrix}\Big) y_1^{2s-1} \frac{dx_2 dy_2}{y_2^2} \frac{dy_1}{y_1}.
\end{equation}
A short matrix computation and the use of a characteristic property of the Whittaker function (Definition 5.4.1 (1) of \cite{Goldfeld}) shows 
\begin{equation}
W_J\left(M \begin{pmatrix} z_2 y_1 & \\ & 1 \end{pmatrix}\right) = e(m_2 x_2) W_J\left(\begin{pmatrix} m_1 |m_2|  &  & \\ & m_1 & \\ & & 1 \end{pmatrix}\begin{pmatrix} y_1 y_2 &  & \\ & y_1 & \\ & & 1 \end{pmatrix}  \right).
\end{equation}
Using this, inserting
the Fourier expansion for $u_j$ \eqref{eq:ujFourier} and evaluating the $x_2$-integral by orthogonality of characters, we have 
\begin{multline}
 \mathcal{L}(s, F \times \overline{u_j}) =  2  \sum_{m_1 \geq 1} \sum_{m_2 \neq 0} \frac{A_F(m_1, m_2) \overline{\rho_j}(m_2)}{m_1 |m_2|^{\half}} 
\int_0^{\infty} \int_0^{\infty}
K_{it_j}(2 \pi |m_2| y_2)
\\
 W_J\left(\begin{pmatrix} m_1 |m_2|  &  & \\ & m_1 & \\ & & 1 \end{pmatrix}\begin{pmatrix} y_1 y_2 &  & \\ & y_1 & \\ & & 1 \end{pmatrix}  \right)   (y_1^2 y_2)^{s-\half}  y_2^{\half} \frac{d y_1 dy_2}{y_1 y_2^2}.
\end{multline}
Note that $\overline{K_{i\tau}}(x) = K_{i\tau}(x)$ for $x > 0$ and $\tau$ real.
Changing variables $y_2 \rightarrow y_2/|m_2|$ and $y_1 \rightarrow y_1/m_1$ and using \eqref{eq:rholambda} gives
\begin{multline}
 \mathcal{L}(s, F \times \overline{u_j}) =  2  \sum_{m_1 \geq 1} \sum_{m_2 \neq 0} \frac{A_F(m_1, m_2) \lambda_j(|m_2|) \overline{\rho_j}(\frac{m_2}{|m_2|})}{(m_1^2 |m_2|)^s} 
\\
\int_0^{\infty} \int_0^{\infty}
K_{it_j}(2 \pi y_2)
 W_J\begin{pmatrix} y_1 y_2 &  & \\ & y_1 & \\ & & 1 \end{pmatrix}  (y_1^2 y_2)^{s-\half}  y_2^{\half} \frac{d y_1 dy_2}{y_1 y_2^2}.
\end{multline}
Recall that we say that $u_j$ is even if $\rho_j(-1) = \rho_j(1)$, and that $u_j$ is odd if $\rho_j(-1) = - \rho_j(1)$.  There are no odd Maass forms for $\SL{3}{Z}$, meaning $A_F(m_1, m_2) = A_F(m_1, -m_2)$ (see \cite{Goldfeld} p.163).  This implies $\mathcal{L}(s, F \times \overline{u_j}) = 0$ if $u_j$ is odd.  For $u_j$ even we simply recall the definition \eqref{eq:Gtaudef} to complete the proof.
\end{proof}
\begin{myprop}
  Suppose $F$ is a $\SL{3}{Z}$ Hecke-Maass form as in \eqref{eq:Fourierexpansion}.  Define
\begin{equation}
\label{eq:LambdaintegralrepEisenstein}
 \mathcal{L}(s, F \times \overline{E}(\cdot, \thalf + i\tau)) = \int_0^{\infty} \int_{\SL{2}{Z} \backslash \mathcal{H}^2} \overline{E}(z_2, \thalf + i\tau) F\begin{pmatrix} z_2 y_1 & \\ & 1 \end{pmatrix} y_1^{2s-1} \frac{dx_2 dy_2}{y_2^2} \frac{dy_1}{y_1}.
\end{equation}
Then
\begin{equation}
\label{eq:LambdaLE}
 \mathcal{L}(s, F \times \overline{E}(\cdot, \thalf + i\tau)) =  \overline{\rho}(1, \thalf + i \tau) L(s, F \times E(\cdot, \thalf + i \tau)) G_{\tau}(s),
\end{equation}
where
\begin{equation}
\label{eq:LFEdef}
 L(s, F \times E(\cdot, \thalf + i \tau)) = \sum_{m_1 \geq 1} \sum_{m_2 \geq 1} \frac{A_F(m_1, m_2) \lambda(m_2, \half + i\tau)}{(m_1^2 m_2)^s}.
\end{equation}
\end{myprop}
The proof is very similar to that of Proposition \ref{prop:GL3GL2RS} so we omit it.
\begin{mycoro}[\cite{JPS1}, \cite{JPS2}, \cite{CP}]
 Let $F$ be as above, and let $\widetilde{F}$ be its dual.  Then $L(s, F \times u_j)$ defined by \eqref{eq:GL3GL2RS} and $L(s, F \times E(\cdot, \half + i\tau)$ defined by \eqref{eq:LFEdef} have analytic continuation to the entire complex plane and satisfy the respective functional equations
\begin{equation}
\label{eq:FujFE}
 \mathcal{L}(s, F \times \overline{u_j}) = \mathcal{L}(1-s, \widetilde{F} \times \overline{u_j}), 
\end{equation}
and
\begin{equation}
\label{eq:FEFE}
 \mathcal{L}(s, F \times \overline{E}(\cdot, \thalf + i \tau)) = \mathcal{L}(1-s, \widetilde{F} \times \overline{E}(\cdot, \thalf + i \tau)).
\end{equation}
\end{mycoro}
An explicit proof of \eqref{eq:FujFE} is given on p.375 of \cite{Goldfeld} (the case of \eqref{eq:FEFE} is similar); it essentially follows the lines of Riemann's proof of the functional equation for the zeta function by changing variables $z_2 \rightarrow \leftexp{t}{z_2}^{-1}$, $y_1 \rightarrow y_1^{-1}$.  The measure is invariant under these changes of variables, which have the effect of replacing $s$ by $1-s$, and replacing $F$ by its dual, noting that $F$ is invariant on the left and right by the Weyl group.

\section{The restriction norm and $L$-functions}
In this section we develop a beautiful formula for the restriction norm $N(F)$ in terms of Rankin-Selberg $L$-functions associated to $F$ convolved with Maass forms and Eisenstein series for $\SL{2}{Z}$.
\begin{mytheo}
\label{thm:normformula}
 Let $F$ be a Hecke-Maass form for $\SL{3}{Z}$, $u_j$ be an orthonormal basis of Hecke-Maass forms for $\SL{2}{Z}$, $E(\cdot, \half + i\tau)$ be the Eisenstein series, and recall the definitions \eqref{eq:LambdaL}, \eqref{eq:LambdaLE}.  Then
\begin{multline}
\label{eq:normformula}
 N(F) = \frac{1}{\pi} \sum_{j \geq 1} \intR |\mathcal{L}(\thalf + it, F \times \overline{u_j})|^2 dt
\\
+ \frac{1}{4 \pi^2} \intR \intR |\mathcal{L}(\thalf + it, F \times \overline{E}(\cdot, \thalf + i\tau)|^2 dt d\tau.
\end{multline}
\end{mytheo}
\begin{proof}
 Since $F$ is a Maass form, it is smooth and has rapid decay and hence $f_{y_1}(z_2) := F\begin{pmatrix} z_2 y_1 & \\ & 1 \end{pmatrix} \in L^2(\SL{2}{Z} \backslash \mathcal{H}^2)$ (as a function of $z_2$, with $y_1>0$ an arbitrary parameter).  
By the spectral theory of $\SL{2}{Z}$, we have
\begin{equation}
 f_{y_1}(z_2) = \sum_{j \geq 0} \langle f_{y_1}, u_j \rangle u_j(z_2) + \frac{1}{4 \pi} \intR \langle f_{y_1}, E(\cdot, \thalf + i \tau) \rangle E(z_2, \thalf + i \tau) d\tau,
\end{equation}
where $u_0$ %= \frac{1}{\sqrt{\text{Vol}(\SL{2}{Z}\backslash \mathcal{H}^2)}}$ 
is the constant eigenfunction.  The computations in the proof of Proposition \ref{prop:GL3GL2RS} 
show that $\langle f_{y_1}, u_0 \rangle = 0$.  Then Parseval's formula reads
\begin{equation}
\int_{\SL{2}{Z} \backslash \mathcal{H}^2} |f_{y_1}(z_2)|^2 d^*z_2 = \sum_{j \geq 1} |\langle f_{y_1}, u_j \rangle|^2 + \frac{1}{4 \pi} \intR |\langle f_{y_1}, E(\cdot, \thalf + i \tau) \rangle |^2  d\tau.
\end{equation}

The Plancherel formula says
\begin{equation}
 \int_0^{\infty} |\phi(y)|^2 \frac{dy}{y} = \frac{1}{\pi} \intR |\widetilde{\phi}(2it)|^2 dt, \quad \text{where} \quad 
% \end{equation}
% where
% \begin{equation}
 \widetilde{\phi}(2it) = \int_0^{\infty} \phi(y) y^{2it} \frac{dy}{y}.
\end{equation}
Therefore,
\begin{equation}
 \int_0^{\infty} |\langle f_{y_1}, u_j \rangle|^2 \frac{dy_1}{y_1} = \frac{1}{\pi} \intR \left| \int_0^{\infty} \int_{\SL{2}{Z} \backslash \mathcal{H}^2} f_{y_1}(z_2) \overline{u_j}(z_2) d^*z_2 \thinspace y_1^{2it} \frac{dy_1}{y_1} \right|^2 dt,
\end{equation}
which in view of \eqref{eq:Lambdaintegralrep} gives
\begin{equation}
 \int_0^{\infty} |\langle f_{y_1}, u_j \rangle|^2 \frac{dy_1}{y_1} = \frac{1}{\pi} \intR |\mathcal{L}(\thalf + it, F \times \overline{u_j})|^2 dt.
\end{equation}
Similarly, we have
\begin{equation}
 \int_0^{\infty} |\langle f_{y_1}, E(\cdot, \thalf + i\tau)\rangle|^2 \frac{dy_1}{y_1} = \frac{1}{\pi} \intR \left| \int_0^{\infty} \int_{\SL{2}{Z} \backslash \mathcal{H}^2} f_{y_1}(z_2) \overline{E}(z_2, \thalf + i\tau) d^*z_2 \thinspace y_1^{2it} \frac{dy_1}{y_1} \right|^2 dt,
\end{equation}
which in view of \eqref{eq:LambdaintegralrepEisenstein} gives
\begin{equation}
 \frac{1}{4\pi} \intR \int_0^{\infty} |\langle f_{y_1}, E(\cdot, \thalf + i\tau)\rangle|^2 \frac{dy_1}{y_1} d\tau = \frac{1}{4 \pi^2} \intR \intR |\mathcal{L}(\thalf + it, F \times \overline{E}(\cdot, \thalf + i\tau)|^2 dt d\tau.
\end{equation}
Recalling the definition \eqref{eq:N(F)} and gathering the above equations completes the proof.
\end{proof}

\section{Exercises with Stirling's approximation}
\label{section:scaling}
Our goal is to use Theorem \ref{thm:normformula} to estimate $N(F)$.  Recall \eqref{eq:LambdaL} and \eqref{eq:LambdaLE}, \eqref{eq:Gtau}, and \eqref{eq:rhoj1} and \eqref{eq:rhoj1E}.  It is necessary to understand the size of $|G_{\tau}(\half + it)|$, as well as the size of the analytic conductor of $L(\half + it, F \times u_j)$.  We perform these computations in this section.  Without this work it is not even obvious how to bound $N(F)$ on the assumption of the Lindel\"{o}f hypothesis for all the $L$-functions under consideration.

Suppose that $F$ is in the tempered spectrum of $\Delta$, which recall means that $\alpha, \beta, \gamma$ given by \eqref{eq:alpha}-\eqref{eq:gamma} are real.
By re-labelling these parameters, we suppose that
\begin{equation}
\label{eq:Langlandsordering}
\gamma \leq \beta \leq \alpha.  
\end{equation}
Combining this ordering with the relation $\alpha + \beta + \gamma =0$ (which follows directly from \eqref{eq:alpha}-\eqref{eq:gamma}), observe the simple inequalities $\alpha + 2\gamma \leq 0 = \alpha + \beta + \gamma \leq 2\alpha + \gamma$ which imply
%\begin{equation}
 $\alpha \leq -2 \gamma$, and $-\gamma \leq 2 \alpha$.  
%\end{equation}
Set 
\begin{equation}
 \label{eq:Tdef}
T= |\alpha| + |\beta| + |\gamma|,
\end{equation}
so that $T \asymp \lambda_F(\Delta)^{1/2}$, and we have $\alpha \asymp |\gamma| \asymp T$.  We use $T$ as a fundamental parameter for the rest of the paper.
It is convenient to observe that in estimating $N(F)$ we may suppose without loss of generality that $\beta \geq 0$.  This follows from the functional equations of the Rankin-Selberg $L$-functions on the right hand side of \eqref{eq:normformula}, which replace $F$ by its dual, switching the signs of the Langlands parameters.

\begin{mylemma}
\label{lemma:meanvaluereduction}
 Suppose that $F$ is an $\SL{3}{Z}$ Hecke-Maass form which is in the tempered spectrum of $\Delta$ with $\gamma \leq 0 \leq \beta \leq \alpha$.  Let $u_j$ and $E(z,s)$ be as in Section \ref{section:background}.  Write
\begin{equation}
 L(s, F \times u_j) = \sum_{n=1}^{\infty} \frac{\lambda_{F \times u_j}(n)}{n^s}, \quad L(s, F \times E(\cdot, \thalf + i\tau)) = \sum_{n=1}^{\infty}  \frac{\lambda_{F \times E_{\tau}}(n)}{n^s}.
\end{equation}
Then there exist $O(\log^2{T})$ tuples $(R,S,D, Q, T_0)$ of real numbers and a smooth function $W$ satisfying the following properties:
\begin{equation}
\begin{split}
\label{eq:RDST}
1 \ll R \ll (\alpha - \beta) + \log^2{T}, \qquad R \ll D \leq S \ll T, 
\\
T_0 \in \{\alpha, \beta, \gamma\}, \qquad Q \asymp T^2 DR (S + (\alpha - \beta)) (1 + (\alpha-\beta)),
\end{split}
\end{equation}
$W$ satisfies \eqref{eq:WAFE}, and 
\begin{multline}
\label{eq:normUB}
 N(F) \ll \sum_{(R,S,D,Q,T_0, \pm)} Q^{-\half+\varepsilon} \Big[\int_{-R}^{R} \sum_{S \leq t_j \leq S + D} 
\alpha_j
%\frac{|\rho_j(1)|^2}{\cosh(\pi t_j)} 
\Big|\sum_{n \geq 1} \frac{\lambda_{F \times u_j}(n) W(n)}{n^{\half + it \pm it_j + iT_0}}\Big|^2 dt 
\\
+ \frac{1}{4 \pi} \int_{-R}^{R} \int_{S}^{S+D} 
%\frac{|\rho(1,\half + i\tau)|^2}{\cosh(\pi \tau)} 
\alpha_{\tau}
\Big|\sum_{n \geq 1}  \frac{\lambda_{F \times E_{\tau}}(n) W(n)}{n^{\half + it \pm i\tau + iT_0}}\Big|^2 d\tau dt \Big]  + O(T^{\varepsilon}).
\end{multline}
Furthermore, if $T_0 = \gamma$ then the following more restrictive relations hold: $D \ll \alpha-\beta$ and $S \asymp T$.
%where the $L$-functions under consideration [write them out] have conductor $\asymp Q$.
\end{mylemma}

It follows from Lemma \ref{lemma:meanvaluereduction} that the Lindel\"{o}f hypothesis implies $N(F) \ll T^{\varepsilon} |A_F(1,1)|^2$.  The content of Theorem \ref{thm:normbound} is thus to remove this unproved hypothesis.  In some cases, such as when $\beta = 0$, $\alpha = -\gamma \asymp T$, $R \asymp T^{1-\varepsilon}$, $S \asymp D \asymp T$, the bound in Theorem \ref{thm:normbound} corresponds to the Lindel\"{o}f hypothesis on average.  In other cases, such as when $\beta=0$, $R \asymp D \asymp S \asymp 1$, the family is very small and we can only claim the convexity bound.

The proof of Lemma \ref{lemma:meanvaluereduction} takes up this section; we prove some intermediate lemmas building up to the full proof of Lemma \ref{lemma:meanvaluereduction}.  At its essence, the proof is simply a long but elementary computation with many cases to consider.

Define the very convenient variables
\begin{equation}
\label{eq:XY}
 X = t + \tau, \qquad Y = t - \tau.
\end{equation}
Observe that $X \geq Y$ iff $\tau \geq 0$, which shall be true in the forthcoming arguments.  Next for any real number $Z$ define
\begin{equation}
 q_{F}(Z) = (1 + |Z-\alpha|)(1 + |Z-\beta|)(1 + |Z-\gamma|),
\end{equation}
and then with $X, Y$ as in \eqref{eq:XY}, set
\begin{equation}
 q_{F}(t,\tau) = q_{F}(X) q_{F}(Y).
\end{equation}
Observe that $L(\thalf + it, F \times u_j)$ and $L(\thalf + it, F \times E(\cdot, \thalf + i \tau)$ both have analytic conductor (as in Chapter 5 of \cite{IK}) $q_{F}(t,\tau)$, where in the $u_j$ case, $\tau = t_j$.

\begin{mylemma}
\label{lemma:Gsize}
Let $\gamma \leq 0 \leq \beta \leq \alpha$, with $\alpha + \beta + \gamma = 0$.  Suppose $X \geq Y$ and let $X_0 \geq 0$.  If $X \not \in [\beta - X_0, \alpha + X_0]$ then
\begin{equation}
\label{eq:GboundX}
 \cosh(\pi \tau) |G_{\tau}(\thalf + it)|^2 \ll \exp(-\pi X_0).
\end{equation}
Similarly, suppose $Y_0 \geq 0$.  If $Y \not \in [\gamma - Y_0, \beta + Y_0]$ then
\begin{equation}
\label{eq:GboundY}
 \cosh(\pi \tau) |G_{\tau}(\thalf + it)|^2 \ll \exp(- \pi Y_0).
\end{equation}
On the other hand, if
\begin{equation}
 \beta - X_0 \leq X \leq \alpha + X_0, 
\qquad
% \end{equation}
% and
% \begin{equation}
 \gamma - Y_0 \leq Y \leq \beta + Y_0,
\end{equation}
then
\begin{equation}
\label{eq:noexpdecay}
 \cosh(\pi \tau) |G_{\tau}(\thalf + it)|^2 \ll q_{F}(t,\tau)^{-\half},
\end{equation}
where the implied constant is absolute.
\end{mylemma}
\begin{proof}
 Recall that Stirling's approximation implies $|\Gamma(\sigma + i v)|^2 \ll (1+|v|)^{-1 + 2 \sigma} \exp(-\pi|v|)$ for $\sigma > 0$ fixed and all $v \in \mr$.  A computation then gives
\begin{equation}
 \cosh(\pi \tau) |G_{\tau}(\thalf + it)|^2 \ll \exp\big(- \frac{\pi}{2} W_{F}(t,\tau)\big) q_{F}(t,\tau)^{-\half},
\end{equation}
where (note $\cosh(\pi \tau) \ll \exp(\frac{\pi}{2} |X-Y|)$)
\begin{multline}
 W_{F}(t,\tau) = -|X-Y| - |\alpha-\beta| - |\alpha-\gamma| - |\beta-\gamma| 
\\
+ |X-\alpha| + |X-\beta| + |X-\gamma| + |Y-\alpha| + |Y-\beta| + |Y-\gamma|.
\end{multline}
Note that $W_F$ is invariant under permutations of $\alpha, \beta, \gamma$, and also invariant under switching $X$ and $Y$.
We first show $W_F(t,\tau) \geq 0$ for all $t, \tau \in \mr$, or equivalently, all $X, Y \in \mr$.  Observe that $W_F$ is piecewise linear and has limit $+\infty$ as $X$ or $Y$ approach $\pm \infty$.  Therefore its minimum occurs at a critical point.  By symmetry (temporarily forgetting our ordering of the Langlands parameters), it suffices to check that $W_F(t,\tau) \geq 0$ at $X = \alpha$.  In this case,
\begin{equation}
 W_{F}(t,\tau) \big|_{X=\alpha} =  |Y-\beta| + |Y-\gamma| - |\beta-\gamma| \geq 0,
\end{equation}
by the triangle inequality.  This gives \eqref{eq:noexpdecay}, as desired.

By a tedious brute-force computation we obtain the following table of values of $\half W_{F}(t,\tau)$ for $\gamma \leq \beta \leq \alpha$.  We only display the ranges with $X \geq Y$; the rest can be obtained quickly observing that $W_F(t,\tau)$ is symmetric in $X$ and $Y$.
\begin{equation}
 \begin{tabular}{l||c|c|c|c}
$\half W_F(t,\tau)$ & $X \leq \gamma$ & $ < X \leq \beta$ & $<X \leq \alpha$ & $\alpha<X$ \\
 \hline \hline 
$Y \geq \alpha$ &  &  &  & \parbox{2cm}{$(X-\beta)$ \\ $ + 2(Y-\alpha)$} \\ 
\hline
$\alpha>Y \geq \beta$ &  &  & $(Y-\beta)$ & \parbox{2cm}{$(X-\alpha)$ \\ $+(Y-\beta)$} \\
\hline
$\beta>Y \geq \gamma$ &  & $ (\beta - X)$ & $0$ & $(X-\alpha)$ \\
\hline
$ \gamma > Y$ & \parbox{2cm}{$2(\gamma-X)$\\ $+ (\beta-Y)$} & \parbox{2cm}{$ (\beta - X)$ \\ $ +(\gamma-Y)$} & $(\gamma-Y)$ & \parbox{2cm}{$(X-\alpha)$ \\ $+ (\gamma-Y)$}
\end{tabular}
\end{equation}
% \begin{equation}
%  \begin{tabular}{l||p{2cm}|p{2cm}|p{2cm}|p{2.0cm}}
% $\half W_F(t,\tau)$ & $X \leq \gamma$ & $ < X \leq \beta$ & $<X \leq \alpha$ & $\alpha<X$ \\
%  \hline \hline 
% $Y \geq \alpha$ &  &  &  & $(X-\beta)$ $+2(Y-\alpha)$ \\ 
% \hline
% $\alpha>Y \geq \beta$ &  &  & $(Y-\beta)$ & $(X-\alpha)$ $+(Y-\beta)$ \\
% \hline
% $\beta>Y \geq \gamma$ &  & $ (\beta - X)$ & $0$ & $(X-\alpha)$ \\
% \hline
% $ \gamma > Y$ & $2(\gamma-X)$ $+ (\beta-Y)$ & $ (\beta - X)$ $ +(\gamma-Y)$ & $(\gamma-Y)$ & $(X-\alpha)$ $+ (\gamma-Y)$
% \end{tabular}
%\end{equation}
In particular, we directly read from the table the bounds \eqref{eq:GboundX} and \eqref{eq:GboundY}.
\end{proof}

In practice, Lemma \ref{lemma:Gsize} says that $N(F)$ is determined by the range $\beta - \log^{2}{T} \leq X \leq \alpha + \log^2{T}$ and $\gamma - \log^{2}{T} \leq Y \leq \beta + \log^2{T}$, say.
The reason is that we may assume $X \geq Y$ in view of the expression \eqref{eq:normformula} which naturally has $t_j \geq 0$ for the discrete spectrum, and by symmetry we may suppose $\tau \geq 0$ in the continuous spectrum. For $X$ or $Y$ outside of this range there is exponential decay in the completed $L$-function.  

We need a still finer dissection of the $X$ and $Y$ ranges in order to fix the size of $q_{F}(t,\tau)$.  
\begin{mylemma}
\label{lemma:qXYsize}
Suppose $\gamma \leq 0 \leq \beta \leq \alpha$, with $\alpha + \beta + \gamma = 0$.  If
 $X = \beta +l$, with $0 \leq l \leq \frac{\alpha-\beta}{2}$, or $X = \alpha - l$ with $0 \leq l \leq \frac{\alpha-\beta}{2}$, then
\begin{equation}
 q_F(X) \asymp T(1+ (\alpha-\beta))(1 + l).
\end{equation}
If $Y = \beta - s$ with $ 0 \leq s \leq \frac{\alpha-\beta}{2}$, then
\begin{equation}
 q_F(Y) \asymp T(1+ (\alpha-\beta))(1 + s).
\end{equation}
If $Y = \beta - s$ with $ \frac{\alpha-\beta}{2} \leq s \leq \frac{\beta-\gamma}{2}$ then
\begin{equation}
 q_F(Y) \asymp T (1 + s)^2.
\end{equation}
If $Y = \gamma + s$ with $|s| \leq \frac{\beta-\gamma}{2}$ then
\begin{equation}
 q_F(Y) \asymp T^2 (1 + |s|).
\end{equation}
If $|Z-\alpha| \leq \log^2{T}$ or $|Z-\beta| \leq \log^2{T}$ %or $|Z-\gamma| \leq \log^2{T}$ 
then
\begin{equation}
 (1 + (\alpha-\beta))T \ll q_F(Z) \ll \log^2{T} (1 + (\alpha-\beta)) T.
\end{equation}
\end{mylemma}
\begin{proof}
 The estimates for $q_F(X)$ follow from a direct computation, using that $\beta \geq 0$ so that $\beta - \gamma \asymp T$.
The estimates for $q_F(Y)$ are similar.  For example, if $0 \leq s \leq \frac{\alpha-\beta}{2}$ then we use
\begin{equation}
 T \gg \beta - \gamma - s \geq \beta - \gamma - \frac{\alpha-\beta}{2} = \half \alpha + \frac52 \beta \gg T,
\end{equation}
which is the key to estimating $q_{F}(Y)$ in this range.  The other ranges are similar.
\end{proof}

Now we are ready to chop up the regions of summation and integration on the right hand side of \eqref{eq:normformula} into managable pieces.  
\begin{mylemma}
\label{lemma:partition}
 There exists a sequence of $O(\log^2{T})$ pairs of real numbers $U$, $V$ and a pair of real numbers $X_1, Y_1$ (depending on $U$, $V$) satisfying $1 \ll U \leq \frac{\alpha - \beta}{4}$, $1 \ll V \ll T$, $\beta + U \leq X_1$, $X_1 + U \leq \alpha - U$, $\gamma + V \leq Y_1$, $Y_1 + V \leq \beta - V$, such that on each interval $I_{X_1, Y_1, U, V}$ defined by $X_1 \leq X \leq X_1 + U$ and $Y_1 \leq Y \leq Y_1 + V$ we have $q_F(t, \tau) \asymp Q$ where $Q$ depends on $F$, $U$, $V$, $X_1$, and $Y_1$ only.  Furthermore, every $X$, $Y$ satisfying $\beta + 1 \leq X \leq \alpha-1$ and $\gamma +1 \leq Y \leq \beta -1$ lies in one of the above intervals.

More precisely, we have formulas for $Q$ depending on the case:
\begin{equation}
\label{eq:Qsize}
 Q \asymp \begin{cases}
           T^2 UV (1 + (\alpha-\beta))^2, \qquad \text{if } Y_1 = \beta -2V, \quad V \leq \frac{\alpha-\beta}{4} \\
	   T^2 UV^2 (1 + (\alpha-\beta)), \qquad \text{if } Y_1 = \beta -2V, \quad \frac{\alpha-\beta}{4} \leq V \leq \frac{\beta - \gamma}{4} \\
	   T^3 U V (1 + (\alpha-\beta)), \qquad \text{if } Y_1 = \gamma+V, \quad V \leq \frac{\beta-\gamma}{4}
          \end{cases}
\end{equation}
Furthermore, $X_1$ equals either $\beta+U$ or $\alpha - 2U$.
\end{mylemma}
\begin{proof}
We consider first the most important cases with $\beta +1 \leq X \leq \alpha-1$ and $\gamma+1 \leq Y \leq \beta -1$.
In view of Lemma \ref{lemma:qXYsize}, the $X$ parameter naturally lies in an interval of the form $X_1 \leq X \leq X_1 + U$, where $\beta + U \leq X_1$, and $X_1 + U \leq \alpha - U$.  Here $U$ runs over $O(\log{T})$ numbers with $1 \leq U \leq \frac{\alpha - \beta}{2}$.  For such $X$, we have $q_{F}(X) \asymp U(1+(\alpha - \beta)) T$.  On the other hand, $Y$ naturally lies in an interval of the form
$\beta - 2V \leq Y \leq \beta - V$ with $V$ running over $O(\log{T})$ dyadic numbers of the form $1 \leq V \leq \frac{\alpha-\beta}{4}$, in which case $q_F(Y) \asymp V (1+(\alpha-\beta))T$, one of the form $\beta - 2V \leq Y \leq \beta - V$ with $\frac{\alpha-\beta}{4} \leq V \leq \frac{\beta-\gamma}{4}$, in which case $q_F(Y) \asymp V^2T$, or $Y$ lies in an interval of the form $\gamma + V \leq Y \leq \gamma + 2V$ with $1 \leq V \leq \frac{\beta-\gamma}{4}$, in which case $q_F(Y) \asymp VT^2$.  The total number of tuples $(U,V, X_1, Y_1)$ is $O(\log^2{T})$.
\end{proof}

We are finally ready to prove Lemma \ref{lemma:meanvaluereduction}.
\begin{proof}
We use Theorem \ref{thm:normformula}.
As shorthand, let $\mathcal{L} = \log{T}$.  Recall that $X = t + \tau$, $Y = t - \tau$, and for the discrete spectrum sum in \eqref{eq:normformula}, $\tau = t_j$.  
By Lemma \ref{lemma:Gsize}, we may restrict the variables $X$ and $Y$ appearing in \eqref{eq:normformula} so that $\beta - \mathcal{L}^2 \leq X \leq \alpha + \mathcal{L}^2$ and $\gamma - \mathcal{L}^2 \leq Y \leq \beta + \mathcal{L}^2$, with an error term of size $O(T^{-100})$, satisfactory for Lemma \ref{lemma:meanvaluereduction}.  

For simplicity, first consider the special case $\beta + 1 \leq X \leq \alpha -1$, $\gamma+1 \leq Y \leq \beta - 1$.
By Lemma \ref{lemma:partition}, we conclude that the contribution to the right hand side of \eqref{eq:normformula} of such $X$ and $Y$ is
\begin{multline}
 \ll \sum_{X_1, Y_1, U, V} Q^{-\half} \Big[\mathop{\int \sum}_{(X,Y) \in I_{X_1, Y_1, U, V}} |L(\thalf + it, F\times u_j)|^2 dt %\right.
\\
%\left. 
+ 
\frac{1}{4 \pi} \mathop{\int \int}_{(X,Y) \in I_{X_1, Y_1, U, V}} |L(\thalf + it, F\times E(\cdot, \thalf + i \tau)|^2 dt d\tau
\Big].
\end{multline}
Next apply the approximate functional equation, Lemma \ref{lemma:AFE}, where the conductor $Q$ is given by \eqref{eq:Qsize} depending on the case under consideration.

Next we unravel the condition that $(X, Y) \in I_{X_1, Y_1, U, V}$ and replace this by conditions on $t$ and $\tau$.  We shall use positivity to separate the dependence of $t$ and $\tau$.  We show that
 each choice of $X_1, Y_1, U, V$ as in Lemma \ref{lemma:partition} leads to an instance of Lemma \ref{lemma:meanvaluereduction}.
We first split into two cases:

Suppose $U > V$.  In this case, we change variables $t \rightarrow t + \tau + Y_1$ to get the summation conditions $0 \leq t \leq V$ and $X_1 \leq t + 2 \tau + Y_1 \leq X_1 + U$.  Thus $X_1 - Y_1 - t \leq 2 \tau \leq X_1 - Y_1 - t + U$.  By positivity, we extend this to $X_1 - Y_1 - V \leq 2 \tau \leq X_1 - Y_1 + U$.

Suppose $U \leq V$.  In this case, we change variables $t \rightarrow t - \tau + X_1$ to get the summation conditions $0 \leq t \leq U$ and $Y_1 \leq t - 2 \tau + X_1 \leq Y_1 + V$.  Thus $X_1 - Y_1 + t - V \leq 2 \tau \leq X_1 - Y_1 + t$.  By positivity, we extend this to $X_1 - Y_1 - V \leq 2 \tau \leq X_1 - Y_1 + U$, which is the same answer as in the previous case.

In both cases we almost obtain an instance of a sum/integral as given on the right hand side of \eqref{eq:normUB}, that is, we have a $t$-integral and a spectral sum/integral with a bilinear form of the shape as given by \eqref{eq:normUB}, but with an extra $v$-integral of length $O(\log{T})$ coming from the approximate functional equation.  However, this $v$-integral can be absorbed into the $t$-integral by positivity (simply change variables $t \rightarrow t-v$, extend the range of $t$ to $|t| \leq 2U$ by positivity and integrate trivially over $v$).
If $U > V$ then $T_0 = Y_1$ and if $U \leq V$ then $T_0 = X_1$.  We claim the following table describes the family in all the cases.  Explanation follows the display of the table.
\begin{equation}
 \begin{tabular}{l||c|c|c|c|c|c|c|c}
\text{case} & $X_1$ & $Y_1$ & family & $T_0$ & $R$ & $D$ & $S$ & $Q$ \\
\hline
1a & $\beta + U$ &  $\beta-2V$ & U+V & $\beta$ & $V$ & $U$ & $U$ & $T^2 UV (1 + (\alpha-\beta))^2$ \\
1b & $\beta + U$ &  $\beta-2V$ & U+V & $\beta$ & $U$ & $V$ & $V$ & $T^2 UV (1 + (\alpha-\beta))^2$ \\
2a & $\alpha - 2U$ &  $\beta-2V$ & $\alpha-\beta - 2U+V$ & $\beta$ & $V$ & $U$ & $\alpha-\beta$ & $T^2 UV (1 + (\alpha-\beta))^2$ \\
2b & $\alpha - 2U$ &  $\beta-2V$ & $\alpha-\beta - 2U+V$ & $\alpha$ & $U$ & $V$ & $\alpha-\beta$ & $T^2 UV (1 + (\alpha-\beta))^2$ \\
3 & $\beta + U$ &  $\beta-2V$ & $U+V$ & $\beta$ & $U$ & $V$ & $V$ & $T^2 UV^2 (1 + (\alpha-\beta))$ \\
4 & $\alpha - 2U$ &  $\beta-2V$ & $\alpha-\beta - 2U+V$ & $\alpha$ & $U$ & $V$ & $V$ & $T^2 UV^2 (1 + (\alpha-\beta))$ \\
5a & $\beta + U$ &  $\gamma +V$ & $\beta-\gamma +U-2V$ & $\gamma$ & $V$ & $U$ & $T$ & $T^3 UV (1 + (\alpha-\beta))$ \\
5b & $\beta + U$ &  $\gamma +V$ & $\beta-\gamma +U-2V$ & $\beta$ & $U$ & $V$ & $T$ & $T^3 UV (1 + (\alpha-\beta))$ \\
6a & $\alpha - 2U$ &  $\gamma +V$ & $\alpha-\gamma -2U-2V$ & $\gamma$ & $V$ & $U$ & $T$ & $T^3 UV (1 + (\alpha-\beta))$ \\
6b & $\alpha - 2U$ &  $\gamma +V$ & $\alpha-\gamma -2U-2V$ & $\alpha$ & $U$ & $V$ & $T$ & $T^3 UV (1 + (\alpha-\beta))$ \\
\end{tabular}
\end{equation}
Here the $X_1$ column denotes whether $X_1 = \beta + U$ or $X_1 = \alpha - 2U$, and it is understood that $1 \ll U \leq \frac{\alpha - \beta}{4}$.  The cases $1a, 1b, 2a, 2b$ have $1 \ll V \leq \frac{\alpha-\beta}{4}$; cases $3$ and $4$ have $\frac{\alpha-\beta}{4} \leq V \leq \frac{\beta-\gamma}{4}$, and $5a, 5b, 6a, 6b$ have $1 \ll V \leq \frac{\beta-\gamma}{4}$.  We use the label $a$ appended to a particular case to denote $U > V$, and likewise $b$ denotes $U \leq V$.  In cases $3$ and $4$ we automatically have $U \leq V$.  If the entry in the ``family'' column is $x$ then this means the spectral sum (or integral) is restricted to $x \leq 2 \tau \leq x + U + V$.  The remaining columns give the values of $T_0$, $R$, $D$, $S$, and $Q$.  The value of $T_0$ requires a comment; above we mentioned that if $U > V$ then $T_0 = Y_1$ while if $U \leq V$ then $T_0 = X_1$.  For the sake of exposition, suppose that $U > V$ and $Y_1 = \beta - 2V$.  Then the $t$-integral is over $0 \leq t \leq V$ so we can change variables $t \rightarrow t + 2V$ and extend the $t$-integral to $-2V \leq t \leq 2V$ by positivity; this procedure has the effect of replacing $Y_1$ by $\beta$.  This procedure can be done in every one of the cases, giving the displayed value of $T_0$.  It is also worth mentioning that the displayed value of $S$ is true up to a multiplicative constant; for example, in case $2a$ the family is $\alpha - \beta - 2U + V \leq 2 \tau \leq \alpha - \beta - U +2V$, which literally gives $S \leq \tau \leq S + D$ with $D = \half(U +V)$ and $S = \half (\alpha - \beta - 2U +V)$, which satisfies $S \asymp \alpha - \beta$ and $D \asymp U$, as stated.

We can read from the table the conditions $R \ll D \ll S \ll T$, $Q \asymp T^2 RD(1 + (\alpha-\beta))(S + (\alpha-\beta))$, and $T_0 \in \{\alpha, \beta, \gamma\}$, as stated in \eqref{eq:RDST}.  Furthermore, if $T_0 = \gamma$ then $S \asymp T$.

Now we briefly sketch how to extend the above analysis to cover the remaining cases with $\beta - \mathcal{L}^2 \leq X \leq \beta + 1$ or $\alpha - 1\leq X \leq \alpha + \mathcal{L}^2$ or $\gamma - \mathcal{L}^2 \leq Y \leq \gamma + 1$ or $\beta - 1 \leq Y \leq \beta + \mathcal{L}^2$.  We can recover these cases from the previous ones by thickening each of the $X$ and $Y$ intervals by length $\mathcal{L}^2$ at the cost of changing the conductor $Q$ by a multiplicative factor of size at most $\mathcal{L}^2$.  This is easily absorbed by the $Q^{\varepsilon}$ in \eqref{eq:WAFE}.
Translating the conditions on $X$ and $Y$ into conditions on $t$, $\tau$, we see that this thickening procedure simply extends the $t$-integral by $O(\mathcal{L}^2)$ and the $\tau$-sum (or -integral) by $O(\mathcal{L}^2)$.  This has the effect of changing the family to one of the form $|t| \leq R + \mathcal{L}^2$, $S - \mathcal{L}^2 \leq 2 \tau \leq S + D + 2\mathcal{L}^2$.  If $S \geq 2\mathcal{L}^2$ then this is already of the form stated in Lemma \ref{lemma:meanvaluereduction}.  If $S \ll \mathcal{L}^2$ then $R \ll \mathcal{L}^2$ too so there essentially is no family to average over.  In this case the convexity bound gives the contribution to $N(F)$ of $T^{\varepsilon}$, as desired.
\end{proof}

\section{Some tools}
\label{section:bilinearforms}
The rest of the paper concerns the estimation of the right hand side of \eqref{eq:normUB}.  We gather here some facts useful in the proof.

\begin{mylemma}
\label{lemma:basicFourier}
 Suppose that $X, Y > 0$ and $r(x)$ is a Schwartz-class function satisfying
\begin{equation}
\label{eq:rderiv}
 |r^{(j)}(x)| \leq C_j Y^{-j} (1 + \frac{|x|}{X})^{-2},
\end{equation}
for some $C_j \geq 0$, for each $j=0, 1, 2, \dots$.  Then
\begin{equation}
 \widehat{r}(y) \ll_j  X(1 + |y|Y)^{-j}.
\end{equation}
\end{mylemma}
\begin{proof}
 Standard integration by parts.
\end{proof}
\begin{mylemma}
 Let $g$ be a fixed smooth function with compact support.  Suppose that for some $Y \geq 1$, $f$ satisfies
\begin{equation}
\label{eq:almostlinear}
 f(0) = 0, \quad f'(0) = 1, \quad f^{(j+1)}(y) \ll Y^{-j}
\end{equation}
for $j=1, 2, \dots$, and all $y$ in the support of $g$.  Define the function $I$ by
\begin{equation}
 I(\lambda) = \intR g(y) e^{i \lambda f(y)} dy.
\end{equation}
Then for any $C>0$, we have
\begin{equation}
\label{eq:Fbound}
 I(\lambda) \ll_{C} (1 + \min(|\lambda|, Y))^{-C}.
\end{equation}
More precisely, $I$ has an asymptotic expansion of the form
\begin{equation}
\label{eq:Fasymp}
 I(\lambda) = I_0(\lambda) + \dots +  I_{K}(\lambda) + 
O(Y^{-K/2})
%O(( \frac{1+|\lambda|}{Y})^{K+1}),
\end{equation}
where each $I_j$ is a function satisfying
\begin{equation}
\label{eq:Ijbound}
 \lambda^l I_j^{(l)}(\lambda) \ll_{j,l, C} Y^{-j} (1 + |\lambda|)^{-C}.
\end{equation}
In particular, $I_0(\lambda) = \widehat{g}(-\frac{\lambda}{2\pi})$.
\end{mylemma}
The conditions \eqref{eq:almostlinear} say that $f$ is approximately linear, indicating that $I$ should approximately equal the Fourier transform of $g$.  The asymptotic expansion for $I$ indicates that this indeed is the case.  The techniques used in the proof are integration by parts, Fourier inversion, and Taylor's theorem.
The proof gives a convenient description for each $I_j$ in \eqref{eq:Ijformula} below.
\begin{proof}
The first step is to show \eqref{eq:Fbound} by repeated integration by parts.  This will allow us to assume that $\lambda$ is not too big compared to $Y$, which facilitates the development of the asymptotic expansion.
 Let $r(y) = f(y) - y$.  By the mean value theorem, $r'(y) \ll Y^{-1}$.  Then 
\begin{equation}
\label{eq:Iformula}
I(\lambda) =  \intR h(y) e^{i \lambda y} dy = \widehat{h}(-\frac{\lambda}{2\pi}), 
\quad
\text{where} \quad h(y) = g(y) e^{i \lambda r(y)}.
\end{equation}
We claim that $h$ satisfies the bounds
\begin{equation}
 h^{(j)}(y) \ll_j (1 + \frac{|\lambda|}{Y})^j.
\end{equation}
This can be verified by induction on $j$ with the stronger hypothesis that for each $j \geq 0$, $h^{(j)}(y) = q_j(y) e^{i \lambda r(y)}$ for some function $q_j$ satisfying $q_j^{(k)}(y) \ll_{j,k} (1 + \frac{|\lambda|}{Y})^{j+k}$.  It is easy to check that $q_{j+1}(y) = q_j'(y) + i \lambda r'(y) q_j(y)$, whence one can prove the desired bounds on $q_{j+1}^{(k)}$ by Leibniz' rule.  Now $h$ satisfies the conditions of Lemma \ref{lemma:basicFourier} with the $Y$ from \eqref{eq:rderiv} replaced by our current $(1 + \frac{|\lambda| }{Y})^{-1}$.  Thus we have for any $j=0,1, \dots$
\begin{equation}
 \widehat{h}(-\frac{\lambda}{2 \pi}) \ll_j (1 + \frac{|\lambda|}{1 + \frac{|\lambda|}{Y}})^{-j}.
\end{equation}
Taking $j$ very large as necessary, we obtain \eqref{eq:Fbound}.

Now we derive the asymptotic expansion \eqref{eq:Fasymp}. 
Let $p = \frac{\lambda}{Y}$.  We may suppose $p$ is small, say $\ll Y^{-1/2}$ since otherwise the main terms of both sides of \eqref{eq:Fasymp} are $O(Y^{-C})$ for any $C > 0$, which is smaller than the stated error term.
We return to the definition of $\widehat{h}(-\frac{\lambda }{2 \pi})$.  We take a Taylor series expansion for $r(y)$ in the form
\begin{equation}
 r(y) = r''(0) \frac{y^2}{2!} + \dots + r^{(K+1)}(0)  \frac{y^{K+1}}{(K+1)!} + O(Y^{-K-1}),
\end{equation}
which gives
\begin{equation}
 h(y) = g(y) e^{i \lambda r''(0) \frac{y^2}{2!}}% e^{i \lambda r^{(3)}(0) \frac{y^3}{3!}} 
\dots e^{i \lambda r^{(K+1)}(0)  \frac{y^{K+1}}{ (K+1)!}} (1 + O(p Y^{-K})).
\end{equation}
%Since $|\lambda|  \leq Y$, this error is of size $O(Y^{-K})$.  
Next we obtain a Taylor expansion for each term in the above product.  For $j \geq 2$, write
%\begin{equation}
$c_j = Y \frac{i r^{(j)}(0)}{j!}$, and note that $c_j \ll Y^{-j+2} \ll 1$.  Then we have
%\end{equation}
\begin{equation}
 e^{i \lambda r^{(j)}(0) \frac{y^j}{j!}} = e^{p c_j y^j} = 1 +  p c_j y^j + (p  c_j)^2 y^{2j}/2! + \dots + (p  c_j)^K y^{jK}/K! + O(p^{K+1}).
\end{equation}
By expanding out these products, we obtain an expansion for $h(y)$ of the form
\begin{equation}
 h(y) = g(y)  \sum_{j \leq K} \sum_{k \leq K^2} c_{j,k} p^j y^k + O(p^{K+1}) + O(pY^{-K}),
\end{equation}
where $c_{j,k}$ are certain complex numbers satisfying $c_{j,k} \ll_{K} 1$ (note $c_{0,0} = 1$).  We obtain an asymptotic expansion for $\widehat{h}(-\frac{\lambda}{2 \pi})$ by inserting the above expansion for $h(y)$ into \eqref{eq:Iformula} as we now explain.  Writing $g_k(y) = y^{k} g(y)$, we have
\begin{equation}
\label{eq:5.18}
 I(\lambda) = \sum_{j \leq K} %I_j(\lambda) 
\sum_{k \leq K^2} c_{j,k} p^j \widehat{g_{k}}(-\frac{\lambda}{2 \pi}) 
+ O(p^{K+1}) + O(pY^{-K}).
\end{equation}
Letting
\begin{equation}
\label{eq:Ijformula}
 I_j(\lambda) = \sum_{k \leq K^2} c_{j,k} p^j \widehat{g_k}(-\frac{\lambda}{2\pi}) = Y^{-j}  \sum_{k \leq K^2} c_{j,k} \lambda^j \widehat{g_k}(-\frac{\lambda}{2\pi}),
\end{equation}
we have that $I_j$ satisfies \eqref{eq:Ijbound}, using Lemma \ref{lemma:basicFourier} for each $\widehat{g_k}$.
Thus \eqref{eq:5.18} is the desired asymptotic expansion, \eqref{eq:Fasymp}.
\end{proof}

\begin{mylemma}[\cite{Gallagher}]
\label{lemma:largesieve}
Let $a_n$ be any sequence of complex numbers, and $T \geq 1$.  Then
\begin{equation}
\int_{-T}^{T} \sum_{b \leq B} \sumstar_{x \shortmod{b}} \Big| \sum_{n \leq N} a_n \e{xn}{b} n^{iy}\Big|^2 dy \ll (B^2 T + N) \sum_{n \leq N} |a_n|^2.
\end{equation}
Furthermore, we have the additive character version
\begin{equation}
\int_{-T}^{T} \sum_{b \leq B} \sumstar_{x \shortmod{b}} \Big| \sum_{n \leq N} a_n \e{xn}{b} \e{yn}{C} \Big|^2 dy \ll (B^2 T + C) \sum_{n \leq N} |a_n|^2.
\end{equation}
\end{mylemma}

The following general result is useful for simplifying large sieve-type inequalities.
\begin{mylemma}
\label{lemma:linear}
Let $N \geq 1$ and suppose $b_m$ is a sequence of complex numbers with $m \leq N$.  
 Let $f(y)$ be a smooth function on $\mr$ such that for some $X > 0$, $Y \geq N^{\varepsilon}$, we have for $|y| \leq 2$, $f$ satisfies %\eqref{eq:almostlinear}.
\begin{equation}
\label{eq:fproperties}
 f(0) = 0, \qquad f'(0) = X, \qquad 
 f^{(j+1)}(y) \ll X Y^{-j}, \text{ for } j \geq 1.
\end{equation}
Then there exists a nonnegative Schwartz-class function $q(y)$ depending on the implied constants appearing in \eqref{eq:fproperties} and $\varepsilon$ only, satisfying
\begin{equation}
\label{eq:qbound}
 x^j q^{(j)}(x) \ll_{j,C} (1 + |x|)^{-C},
\end{equation}
such that
\begin{equation}
\label{eq:9.2}
 \int_{-1}^{1} %\sum_{b \leq B} \sum_{\substack{x \shortmod{b} \\ (x,b) = 1}} 
\Big| \sum_{m \leq N} b_m  e(m f(y)) \Big|^2 dy 
\leq \intR q(y)
\Big| \sum_{m \leq N} b_m e(m Xy) \Big|^2 dy + O(N^{-100} \sum_{m \leq N} |b_m|^2).
\end{equation}
\end{mylemma}
The point is that the potentially complicated function $f(y)$ is essentially replaced by its best linear approximation. 
\begin{proof}
Let $g$ be a smooth compactly-supported nonnegative function satisfying $g(y) = 1$ for $|y| \leq 1$, and $g(y) = 0$ for $|y| \geq 2$.  Then
\begin{equation}
\label{eq:5.8}
 \int_{-1}^{1} \Big| \sum_{m \leq N} b_m e(m f(y))\Big|^2 dy \leq \sum_{m,n \leq N} b_m \overline{b_n} \int_{-2}^{2} g(y) e((m-n) f(y)) dy.
\end{equation}
Let $\lambda =  2 \pi X(m-n)$ and set $f_{X}(y) = X^{-1} f(y)$, so that the inner integral is  
\begin{equation}
I(\lambda) = \int_{-2}^{2} g(y) e^{i \lambda f_{X}(y)} dy, 
\end{equation}
where $f_X$ satisfies \eqref{eq:almostlinear}.  Next we insert the asymptotic expansion \eqref{eq:Fasymp} into \eqref{eq:5.8}, so
\begin{equation}
 \int_{-1}^{1} \Big| \sum_{m \leq N} b_m e(m f(y))\Big|^2 dy \leq \sum_{m,n \leq N} b_m \overline{b_n} \sum_{j \leq K} I_j(-\frac{\lambda}{2 \pi}) + O(N Y^{-K/2} \sum_{m \leq N} |b_m|^2).
\end{equation}
Then take $q_K(y) = \sum_{j \leq K} \widehat{I_j}(-y)$ and $\frac{202}{\varepsilon}  \leq K < \frac{202}{\varepsilon} + 1$ so that $q_K$ satisfies \eqref{eq:qbound} (using Lemma \ref{lemma:basicFourier}), and
\begin{equation}
 \int_{-1}^{1} \Big| \sum_{m \leq N} b_m e(m f(y))\Big|^2 dy \leq \sum_{m,n \leq N} b_m \overline{b_n} \intR q_K(y) e^{i \lambda y} dy + O(N^{-100} \sum_{m \leq N} |b_m|^2).
\end{equation}
Using the definition of $\lambda$ and re-separating the variables $m$ and $n$, we obtain
\begin{equation}
 \int_{-1}^{1} \Big| \sum_{m \leq N} b_m e(m f(y))\Big|^2 dy \leq 
 \intR q_K(y) \Big|\sum_{m \leq N} b_m e(mXy) \Big|^2 dy
   + O(N^{-100} \sum_{m \leq N} |b_m|^2).
\end{equation}
If $q_K(y)$ is nonnegative then the proof is complete taking $q(y) = q_K(y)$; otherwise we construct a nonnegative Schwartz-class function $q(y) \geq q_K(y)$.  One such construction proceeds by defining real numbers $M_n := \sup_{n-1 \leq |y| \leq n} |q_K(y)|$, for $n=1,2,3, \dots$.  Note that for each $N > 0$, there exists $C_N$ such that $M_n \leq C_N n^{-N}$.  Then define $q(y) = e \sum_{n \geq 1} M_n e^{-(y/n)^2}$, which dominates $q_K$, and is Schwartz-class.

% Let $q(y) = \sum_{j,k} c_{j,k} h_{j,k}(y)$ so that $q$ is a Schwartz-class function satisfying
% \begin{equation}
%  x^j q^{(j)}(x) \ll_{j,C, \varepsilon} (1 + |x|)^{-C}.
% \end{equation}
\end{proof}

The following Lemma is useful for converting between multiplicative and additive  characters in a bilinear form setting.  The idea used in the proof can be used very generally with various integral transforms.  Indeed, the ideas shall be used later in a more complicated situation in the proof of Lemma \ref{lemma:Phiproperties}.
\begin{mylemma}
\label{lemma:conversion}
 Let $b_m$ be complex numbers, and suppose $T \geq M^{\varepsilon}$ for some $\varepsilon > 0$.  Then
\begin{equation}
\label{eq:multtoadd}
 \int_{-T}^{T} \Big|\sum_{M < m \leq 2M} b_m m^{it}\Big|^2 dt \ll  \int_{|y| \ll T}  \Big|\sum_{M < m \leq 2M} b_m \e{my}{M}\Big|^2 dy + O_{\varepsilon}(M^{-100} \sum_{M < m \leq 2M} |b_m|^2),
\end{equation}
where the implied constants depend on $\varepsilon > 0$ only.
Similarly,
\begin{equation}
\label{eq:addtomult}
 \int_{-T}^{T} \Big|\sum_{M < m \leq 2M} b_m \e{my}{M} \Big|^2 dy \ll  \int_{|y| \ll T}  \Big|\sum_{M < m \leq 2M} b_m m^{it}\Big|^2 dt + O_{\varepsilon}(M^{-100} \sum_{M < m \leq 2M} |b_m|^2),
\end{equation}
\end{mylemma}
\begin{proof}
 The idea is basically a continuous analog of the more well-known conversion between additive and multiplicative characters using Gauss sums.  We shall prove only \eqref{eq:multtoadd}, the other case \eqref{eq:addtomult} being very similar.

Let $g$ be a smooth, nonnegative, even function such that $g(x) \geq 1$ for $|x| \leq 1$, and such that the Fourier transform of $g$ has compact support.  Similarly, let $w(x)$ be a smooth nonnegative function supported on $(0, \infty)$ satisfying $w(x) = 1$ for $1 \leq x \leq 2$.  As a minor convenience we furthermore suppose $w(x) \leq g(x)$.  Then the left hand side of \eqref{eq:multtoadd} is
\begin{equation}
 \leq \intR g(t/T) \Big|\sum_{m} b_m w(m/M) m^{it}\Big|^2 dt =:J,
\end{equation}
where we assume for convenience that $b_m$ is supported on $M < m \leq 2M$.  By the Fourier inversion theorem,
\begin{equation}
 w(x/M) x^{it} = \intR \widehat{f_t}(y) e(xy) dy, \qquad \widehat{f_t}(y) = \intR w(x/M) x^{it} e(-xy) dx.
\end{equation}
An integration by parts argument shows that if $|y| M \gg T$ with a large enough implied constant depending on the support of $w$, we have for any $C > 0$
\begin{equation}
 \widehat{f_t}(y) \ll_C T^{-C}.
\end{equation}
Then with $Y = \frac{T}{M}$, we have
\begin{equation}
 J = \intR g(t/T) \Big|\sum_{m} b_m \int_{|y| \ll Y} \widehat{f_t}(y) e(my) dy\Big|^2 dt + O(M^{-100} \sum_m |b_m|^2),
\end{equation}
taking $C$ large enough with respect to $\varepsilon$.  Write this expression for $J$ as $J_1$ plus the error term.

Now we open up the square to get
\begin{equation}
 J_1 = \sum_{m,n} b_m \overline{b_n} \int_{|y_1| \ll Y} \int_{|y_2| \ll Y} e(my_1 - ny_2) \Big[ \intR g(t/T) \widehat{f_{t}}(y_1) \overline{\widehat{f_t}}(y_2) dt\Big] dy_1 dy_2.
\end{equation}
Using the definition of $\widehat{f_t}$, this $t$-integral takes the form
\begin{equation}
 \intR \intR w(x_1/M) w(x_2/M) e(-x_1 y_1 + x_2 y_2) \intR g(t/T) (x_1/x_2)^{it} dt dx_1 dx_2.
\end{equation}
This innermost $t$-integral can be expressed as $T\widehat{g}(\frac{T}{2\pi } \log(x_2/x_1))$, where recall $\widehat{g}$ has compact support, and where $x_1, x_2 \asymp M$ from the support of $w$.  Thus the integral is zero unless $|x_1-x_2| \ll M/T \asymp Y^{-1}$.  We impose this condition on $x_1$ and $x_2$, and again write $J_1$ as a double sum and a quintuple integral as follows
\begin{multline}
\label{eq:J1quintuple}
 J_1 = \intR g(t/T) \mathop{\int \int}_{|x_1-x_2| \ll Y^{-1}} w(x_1/M) w(x_2/M) (x_1/x_2)^{it} 
\\
\Big(\sum_m \int_{|y_1| \ll Y} b_m e(my_1) e(-x_1 y_1) dy_1 \Big) 
\Big(\sum_n \int_{|y_2| \ll Y} \overline{b_n} e(-ny_2) e(x_2 y_2) dy_2\Big) dx_1 dx_2 dt.
\end{multline}
We put in absolute value signs to write this in the form $|J_1| \leq \int_{t} \int_{x_1} \int_{x_2} |\sum_m \int_{y_1} | |\sum_n \int_{y_2}|$, and then apply 
the simple inequality $|A| |B| \leq \half(|A|^2 + |B|^2)$.  In our application, each of these two terms lead to the same sum, so we have
\begin{equation}
 |J_1| \leq \intR g(t/T) \mathop{\int \int}_{|x_1-x_2| \ll Y^{-1}} w\leg{x_1}{M} w\leg{x_2}{M} \Big|\sum_m \int_{|y_1| \ll Y} b_m e(my_1) e(-x_1 y_1) dy_1 \Big|^2 dx_1 dx_2 dt.
\end{equation}
We easily bound the $t$ and $x_2$ integrals with absolute values, obtaining
\begin{equation}
 |J_1| \ll \frac{T}{Y} \intR w(x/M) \Big|\sum_m \int_{|y| \ll Y} b_m e(my) e(-xy) dy\Big|^2 dx.
\end{equation}

By comparison to \eqref{eq:J1quintuple}, the gain is that we have executed two of the integrals .  The next step is to do essentially the same procedure as before to execute the inner $y$-integral.  Recalling the assumption $w(x) \leq g(x)$, we have after opening the square
\begin{equation}
 |J_1| \ll \frac{T}{Y}  \sum_{m} \sum_n b_m \overline{b_n} \int_{|y_1| \ll Y} \int_{|y_2| \ll Y} e(my_1 - ny_2) \intR g(x/M) e(-x(y_1 - y_2)) dx dy_1 dy_2.
\end{equation}
The inner $x$-integral is $M \widehat{g}(M(y_1 - y_2))$, so we may suppose $|y_1 - y_2| \ll M^{-1}$, since otherwise the $x$-integral is zero.  By a similar arrangement as in the previous paragraph, we have
\begin{equation}
 |J_1| \ll \frac{T}{Y} \intR g(x/M) \mathop{\int \int}_{\substack{|y_1 - y_2| \ll M^{-1} \\ |y_1|, |y_2| \ll Y }} \big|\sum_m b_m e(m y_1) \big|^2 dy_1 dy_2 dx.
\end{equation}
Bounding the $x$- and $y_2$-integrals trivially, we have
\begin{equation}
 |J_1| \ll \frac{T}{Y} \int_{|y| \ll Y} |\sum_m b_m e(my)|^2 dy.
\end{equation}
Changing variables $y \rightarrow \frac{Y}{T} y$ and recalling $Y = T/M$  completes the proof.
\end{proof}

\section{The mean-value results}
With notation given as in Lemma \ref{lemma:meanvaluereduction}, let
\begin{multline}
\label{eq:Mdef}
\mathcal{M}(R,S,D,Q) = \int_{-R}^{R} \sum_{S \leq t_j \leq S + D} \alpha_j \Big|\sum_{n \geq 1} \frac{\lambda_{F \times u_j}(n) W(n)}{n^{\half + it + it_j + iT_0}}\Big|^2 dt 
\\
+  \int_{-R}^{R} \frac{1}{4 \pi} \int_{S}^{S+D} 
\alpha_{\tau}
 \Big|\sum_{n \geq 1} \frac{\lambda_{F \times E_{\tau}}(n) W(n)}{n^{\half + it + i\tau + iT_0}}\Big|^2 d\tau dt.
\end{multline}

Our main technical result is
\begin{mytheo}
\label{thm:mainthm}
 We have
\begin{equation}
\label{eq:mainthm}
  \mathcal{M}(R,S,D,Q) \ll  Q^{\half} |A_F(1,1)|^2 T^{\varepsilon},
\end{equation}
where the implied constant is independent of $F$.
\end{mytheo}
In view of Lemma \ref{lemma:meanvaluereduction}, Theorem \ref{thm:mainthm} immediately implies Theorem \ref{thm:normbound}.

The outline of the proof of Theorem \ref{thm:mainthm} is similar to \cite{Y} but virtually all the details are changed for a variety of reasons.  The main issue is that the $GL_{3}$ form $F$ is varying and it is seemingly very difficult to alter the proof given in \cite{Y} to handle this more general case.  Instead, we found new arguments that are fairly ``soft'' compared to \cite{Y}.  In fact, we were able to avoid any applications of stationary phase or elaborate asymptotic expansions of integral transforms, and instead use only integration by parts.

In this section we perform some simplifications and apply the Kuznetsov formula.
\begin{mylemma}
 Let $\sum_{M} \gamma(n/M) = 1$ be a smooth dyadic partition of unity; that is, $\gamma$ is a certain smooth function with support inside the interval $[1/2,1]$, and $M$ runs over powers of $2$.  Define for any sequence of complex numbers $a_n$,
\begin{multline}
\label{eq:Mandef}
\mathcal{M}(R,S, D,Q, M; a_n) = \int_{-R}^{R} \sum_{S \leq t_j \leq S+D} \alpha_j
\Big| \sum_{M/2<n \leq M} 
a_n \lambda_j(n) n^{-it - it_j} 
\Big|^2 dt
\\
+
\int_{-R}^{R} \frac{1}{4 \pi} \int_{S}^{S+D} \alpha_\tau
\Big| \sum_{M/2<n \leq M} 
a_n \lambda(n, \thalf + i\tau) n^{-it - i\tau} 
\Big|^2 d\tau dt.
\end{multline}
Then with $a_n = a_{n,l,M}$ defined by
\begin{equation}
\label{eq:andef}
a_n = \frac{A_F(l,n)}{n^{\half + iT_0}} W(nl^2) \gamma(n/M),
\end{equation}
we have with $N = Q^{\half + \varepsilon}$,
\begin{equation}
\label{eq:MandMl}
 \mathcal{M}(R, S, D, Q) \ll (\log{N})^3 \sup_{1 \ll M \ll N} \sum_{l \leq \sqrt{N/M}}  l^{-1} 
\mathcal{M}(R,S, D,Q, M; a_n).
\end{equation}
\end{mylemma}
\begin{proof}
First we insert the definitions
\begin{equation}
 \lambda_{F \times u_j}(m) = \sum_{l^2 n = m} \lambda_j(n) A_F(l,n), \qquad \lambda_{f \times E_{\tau}}(m) = \sum_{l^2 n = m} \lambda(n, \thalf + i\tau) A_F(l,n)
\end{equation}
into \eqref{eq:Mdef}.  We remark that it is tempting to think of the sum over $l$ as almost bounded since the $n$-dependence is much more difficult than the behavior with respect to $l$.  For this reason, we use Cauchy's inequality in the form
\begin{equation}
|\sum_{n l^2 \leq N} l^{-1} c_{l,n}|^2 \leq \log{N} \sum_{l \leq \sqrt{N}} l^{-1} | \sum_{n \leq l^{-2} N} c_{l,n}|^2.
\end{equation}
Thus we obtain, with $N = 2 Q^{\half + \varepsilon}$
\begin{multline}
\mathcal{M}(R,S,D,Q) \ll \log{Q} \sum_{l \leq \sqrt{N}} l^{-1} 
\int_{-R}^{R} \Big[\sum_{S \leq t_j \leq S +D} 
\alpha_j
%\frac{|\rho_j(1)|^2}{\cosh(\pi t_j)} 
 \Big| \sum_{n \leq l^{-2} N} \frac{\lambda_j(n) A_F(l,n)}{n^{\half + it + it_j + iT_0}} W(n l^2) \Big|^2 
\\
+ 
 \frac{1}{4 \pi} \int_{S}^{S +D} 
\alpha_{\tau}
%\frac{|\rho(1, \thalf + i\tau)|^2}{\cosh(\pi \tau)}  
\Big| \sum_{n \leq l^{-2} N} \frac{\lambda(n, \thalf + i\tau) A_F(l,n)}{n^{\half + it + it_j + iT_0}} W(n l^2) \Big|^2 d\tau \Big] dt.
\end{multline}
We apply the partition of unity to the inner sum over $n$ above, %in \eqref{eq:MRQZldef}, 
with $M$ restricted to $1 \ll M \ll l^{-2} N$.  Then we apply Cauchy's inequality to this sum over $M$, getting that $\mathcal{M}(R,S,D,Q)$ is
\begin{multline}
 %\mathcal{M}(R,S, D, Q) 
\ll (\log{N})^2 \sum_{M} \sum_{l \leq \leg{N}{M}^{\half}} l^{-1}  \int_{-R}^{R} \Big[ \sum_{S \leq t_j \leq S+D} \alpha_j
\Big| \sum_{n \leq l^{-2} N} \frac{\lambda_j(n) A_F(l,n)}{n^{\half + it + it_j + iT_0}} W(n l^2) \gamma(\frac{n}{M}) \Big|^2 
\\
+
 \frac{1}{4 \pi} \int_{S}^{S+D} \alpha_\tau
\Big| \sum_{n \leq l^{-2} N} \frac{\lambda(n, \thalf + i\tau) A_F(l,n)}{n^{\half + it + i \tau + iT_0}} W(n l^2) \gamma(\frac{n}{M}) \Big|^2 d\tau \Big] dt.
\end{multline}
Bounding this sum over $M$ by the number of terms, $O(\log{N})$ times the supremum over all $1 \ll M \ll l^{-2} N$ completes the proof.
\end{proof}

We do not exploit the sum over $l$ until the very final steps (see the remarks following \eqref{eq:N2*def}) and the reader who considers only the case $l=1$ does not miss many crucial changes from the general case.

Next we state a crude bound that is sufficient only in some extreme cases.
\begin{mylemma}
\label{lemma:Iwanieclargesieve}
 Suppose \eqref{eq:RDST} holds.  Then for any complex numbers $a_n$, we have
\begin{equation}
 \mathcal{M}(R,S,D,Q, M;a_n) \ll R(SD + M) (MT)^{\varepsilon} \sum_{n \leq M} |a_n|^2.
\end{equation}
\end{mylemma}
This bound is acceptable for proving Theorem \ref{thm:mainthm} for $R \ll T^{\varepsilon}$; it is also strong if $M$ happens to be small.  
\begin{proof}
 This follows from a variant of Iwaniec's spectral large sieve inequality \cite{IwaniecLargeSieve} in the form given by Theorem 3.3 of \cite{MotohashiBook}
\begin{multline}
 \sum_{S \leq t_j \leq S + D} \alpha_j \Big| \sum_{n \leq M} \lambda_j(n) a_n \Big|^2
+ \frac{1}{4 \pi} \int_{S}^{S+D} \alpha_{\tau} \Big| \sum_{n \leq M} \lambda(n, \thalf + i \tau) a_n \Big|^2 d\tau 
\\
\ll (SD + M)(SM)^{\varepsilon}\sum_{n \leq M} |a_n|^2.
\end{multline}
We apply this bound to \eqref{eq:Mandef} and integrate trivially over $t$.
\end{proof}
For a technical reason, it is convenient to have $D \ll S T^{-\eta}$ for some $\eta > 0$.
\begin{mylemma}
\label{lemma:Dsize}
 Suppose that \eqref{eq:RDST} holds.  For $\eta > 0$, there exists $S' \asymp S$ such that with $D' =D/T^{\eta}$ we have
\begin{equation}
 \mathcal{M}(R,S,D,Q, M;a_n) \ll T^{\eta} \mathcal{M}(R, S', D', Q, M;a_n).
\end{equation}
\end{mylemma}
\begin{proof}
 This follows simply by breaking up the interval $[S, S+D]$ into subintervals $[S, S+D']$, $[S+D', S+2D']$, \dots, $[S+KD', S + (K+1)D']$ where $T^{\eta} -1 < K \leq T^{\eta}$, and bounding $\mathcal{M}(R,S,D,Q)$ by the number of such subintervals times the bound from the subinterval with the largest contribution.
\end{proof}
Remark.  Lemmas \ref{lemma:Iwanieclargesieve} and \ref{lemma:Dsize} allow us to assume $M \gg T^{\varepsilon}$ and replace the assumptions \eqref{eq:RDST} by
\begin{equation}
\label{eq:RSDT'}
\begin{split}
 T^{\varepsilon} \ll R \ll \alpha - \beta, \quad R \ll D T^{\eta} \ll S \ll T, 
\\
T_0 \in \{\alpha, \beta, \gamma\}, \quad Q \asymp T^2 DR(S + (\alpha-\beta))(1 + (\alpha-\beta))).
\end{split}
\end{equation}

Now we state our overall goal for this section.  Compare it to Lemma 7.1 of \cite{Y}.
\begin{mytheo}
\label{thm:separationofvars}
Suppose that $\varepsilon > 0$, \eqref{eq:RSDT'} holds, $T^{\eta} \ll M \leq T^{100}$, and $a_n$ are arbitrary complex numbers.  Then for some smooth, nonnegative Schwartz-class function $g$, %$g$ satisfying $g(y) = 1$ for $|y| \leq 1$,
we have
\begin{multline}
\mathcal{M}(R,S, D, Q, M; a_n) \ll R SD   \sum_{n \leq M} |a_n|^2 
+   
\\
RS  \intR  g(v/T^{\varepsilon}) 
\sum_{ab \leq \frac{M T^{\varepsilon}}{S R}} \frac{1}{ab} \sumstar_{r \shortmod{b}} 
\Big|\sum_{n \leq M} a_n \e{nr}{b} \e{n v}{abD}  \Big|^2 dv .
\end{multline}
The implied constant depends on $g$, $\varepsilon$, and $\eta$.
\end{mytheo}
The rest of this section is devoted to proving this result.  It follows from the Kuznetsov formula.  Some remarks about the form of the right hand side are in order.  The most important point is that the unknown Hecke eigenvalues of the Maass forms are gone and replaced with explicit exponentials, and the right hand side is a bilinear form. 
An important point is to explain the truncation point 
\begin{equation}
\label{eq:ctruncation}
c \leq \frac{M}{SR} T^{\varepsilon}.
\end{equation}
A reader familiar with the Kuznetsov formula might expect $c$ to be instead truncated at $\frac{MT^{\varepsilon}}{SD}$ which can be much smaller (say if $R$ is close to $1$ and both $S$ and $D$ are close to $T$).  This is true, however, one would obtain a weight function with a phase of shape $\e{2 \sqrt{mn}}{c}$ and there would be an extra cost associated with separating the variables $m$ and $n$.

Let $g$ be a fixed nonnegative Schwartz function satisfying $g(x) = 1$ for $|x| \leq 1$, and whose Fourier transform has compact support.  Then by positivity,
\begin{equation}
 \mathcal{M}(R,S,D,Q, M;a_n) = \int_{-R}^{R} [\dots] dt \leq \intR g(t/R) [\dots] dt,
\end{equation}
where $[\dots]$ indicates the inner sums on the right hand side of \eqref{eq:Mandef}.

Let 
\begin{equation}
\label{eq:Pdef}
 P(r) = \Big(\frac{r^2 + \leg{1}{2}^2}{S^2} \Big) \Big(\frac{r^2 + \leg{3}{2}^2}{S^2} \Big) \dots \Big(\frac{r^2 + \leg{299}{2}^2}{S^2} \Big).
\end{equation}
By positivity, we attach the nonnegative weight $\exp(-(\tau-S)^2/D^2) P(\tau)$, 
 to the spectral sum (and integral) and then relax the truncation on $\tau$, getting
\begin{multline}
\mathcal{M}(R,S, D, Q, M;a_n) \ll \intR g(t/R) 
\Big[
\sum_{t_j > 0} \alpha_j  \exp(-\frac{(t_j-S)^2}{D^2}) P(\tau) \Big| \sum_{n \leq M} \frac{\lambda_j(n) a_n}{n^{it + it_j}} \Big|^2 
\\
+ \frac{1}{4 \pi} \int_{\tau > 0} \alpha_{\tau}  \exp(-\frac{(\tau-S)^2}{D^2}) P(\tau)  \Big| \sum_{n \leq M} \frac{\lambda(n, \thalf + i \tau) a_n}{n^{it + i \tau}} \Big|^2 d \tau
\Big]
dt.
\end{multline}
Let
\begin{equation}
h_{m,n}(r) = \frac{\sinh(r(\pi + i \log\frac{m}{n}))}{\sinh^2(\pi r)} P(r) ( e^{\pi r} \exp(-\frac{(r-S)^2}{D^2}) - e^{-\pi r} \exp(-\frac{(-r-S)^2}{D^2})),
\end{equation}
so that $h_{m,n}$ is even and has rapid decay for $r$ large.  A computation shows for $r > 0$ that
\begin{equation}
\label{eq:hmnproperty}
h_{m,n}(r) = 2 \exp(-\frac{(r-S)^2}{D^2}) P(r) (\leg{m}{n}^{ir} + O(e^{-2\pi r}) ).
\end{equation}
Thus we have 
\begin{equation}
\mathcal{M}(R,S, D, Q, M;a_n) \ll \sum_{m,n \asymp M} \overline{a_m} a_n
\Big[\big(\intR g(t/R) \leg{m}{n}^{it} dt \big) \mathcal{K}(m,n) + O(T^{-100})\Big],
\end{equation}
where
\begin{equation}
\mathcal{K}(m,n) = \sum_{t_j}  \alpha_j  h_{m,n}(t_j) \lambda_j(m) \lambda_j(n) + \frac{1}{4 \pi} \intR \alpha_{\tau} h_{m,n}(\tau) \lambda(m, \thalf + i \tau) \overline{\lambda}(n, \thalf + i \tau) d\tau.
\end{equation}
Notice that the $t$-integral is simply expressed in terms of the Fourier transform of $g$, $\widehat{g}(x) = \intR g(y) e(-xy) dy$.  Using this, and Cauchy's inequality on the error term, we have
\begin{equation}
\label{eq:MRQZl55}
\mathcal{M}(R,S, D, Q, M;a_n) \ll R  \sum_{m,n \asymp M} \overline{a_m} a_n
\widehat{g}\Big(\frac{R}{2 \pi} \log(\frac{n}{m})\Big) \mathcal{K}(m,n) + O(M T^{-100}\sum_{n \leq M} |a_n|^2).
\end{equation}
Since $\widehat{g}$ was assumed to have compact support, then we may assume $|\log(n/m)| \ll R^{-1}$ (with an absolute implied constant).  Equivalently,
\begin{equation}
\label{eq:ghatmnclose}
\frac{|m-n|}{m} \ll R^{-1}.
\end{equation}
We shall impose this condition in the following calculations of $\mathcal{K}(m,n)$.
We make a detour in our proof of Theorem \ref{thm:separationofvars} to understand the integral transform in the Kuznetsov formula as follows.
\begin{mylemma}
\label{lemma:Kmncalc}
Suppose that %$T^{\varepsilon} \ll R \ll D T^{\varepsilon} \ll S \ll T$ 
\eqref{eq:RSDT'} 
and \eqref{eq:ghatmnclose} hold.  Then
\begin{equation}
\mathcal{K}(m,n) = H_0 \delta_{m,n} + \sum_{\pm} \sum_{c \leq \frac{MT^{\varepsilon}}{SR}} c^{-1} S(m,n;c) H_{\pm}\leg{4\pi \sqrt{mn}}{c} + O(T^{-100}) + O(SDM^{-1}),
\end{equation}
where
\begin{equation}
\label{eq:H0bound}
H_0 \ll SD,
\end{equation}
and with $k(r) = \frac{4}{\pi^2} (1 + \frac{D}{S} r) P(S+Dr)\exp(-r^2)$, with $P$ is given by \eqref{eq:Pdef}, we have
\begin{equation}
\label{eq:Hformula}
 H_{\pm}\leg{4\pi \sqrt{mn}}{c} = S \int_{|v| \leq T^{\varepsilon}} \widehat{k}(-\frac{v}{2 \pi}) e^{2i \frac{S}{D} v} \e{\pm m \exp(-v/D)}{c} \e{\pm n \exp(v/D)}{c} dv.
\end{equation}
Furthermore, $H_{\pm}\leg{4\pi \sqrt{mn}}{c} \ll T^{-400}$ unless \eqref{eq:ctruncation} holds.
% \begin{equation}
% \label{eq:ctruncation}
% c \ll_{\varepsilon} \frac{M}{S R} T^{\varepsilon}.
% \end{equation}
\end{mylemma}
\begin{proof}[Proof of Lemma \ref{lemma:Kmncalc}]
The Kuznetsov formula, Theorem \ref{thm:Kuznetsov}, expresses $\mathcal{K}(m,n)$ as a diagonal term plus a sum of Kloosterman sums.  The diagonal term given by $H_0 \delta_{m,n}$ with
\begin{equation}
 H_0 = \pi^{-2} \intR r \tan(\pi r) h_{m,n}(r) dr
\end{equation}
is trivially bounded by \eqref{eq:H0bound}.

The sum of Kloosterman sums takes the form $\sum_{c \geq 1} c^{-1} S(m,n;c) H\leg{4 \pi \sqrt{mn}}{c}$, where
\begin{equation}
\label{eq:Hformula2}
H(x) = %\frac{2i}{\pi} \intR r h(r) \frac{J_{2ir}(x)}{\cosh(\pi r)} dr = 
\frac{2i}{\pi} \int_0^{\infty} r h_{m,n}(r) \frac{J_{2ir}(x) - J_{-2ir}(x)}{\cosh(\pi r)} dr.
\end{equation}
We first require a crude bound on $H(x)$ for small values of $x$ so that we may truncate the $c$-sum.  To this end, we now show
\begin{equation}
\label{eq:Hboundsmallx}
 H(x) \ll \leg{x}{S}^{200} DS.
\end{equation}
By the following integral representation of the $J$-Bessel function (\cite{GR}, 8.411.4),
\begin{equation}
 J_{\nu}(x) = 2 \frac{(\frac{x}{2})^{\nu}}{\Gamma(\nu + \half) \Gamma(\half)} \int_0^{\pi/2} \sin^{2 \nu}{\theta} \cos(x \cos{\theta}) d\theta,
\end{equation}
valid for $\text{Re}(\nu) > -\half$, one derives from a trivial bound and Stirling's approximation that
\begin{equation}
 J_{2iy + 200}(x) \ll \leg{x}{1+|y|}^{200} e^{\pi |y|}.
\end{equation}
Now in \eqref{eq:Hformula2} (actually we need the variant integral over $\mr$; see \eqref{eq:Hdef}) we shift the contour to $\text{Im}(r) = -100$ without crossing any poles (since $P$ defined by \eqref{eq:Pdef} vanishes at the zeros of $\cosh(\pi r)$).  Using the bound
\begin{equation}
 h_{m,n}(-100i + y) \ll \leg{1+|y|}{S}^{302} \exp(-\frac{(y-S)^2}{D^2}),
\end{equation}
we immediately obtain \eqref{eq:Hboundsmallx}.

Using \eqref{eq:Hboundsmallx} for $x \leq M^{-1}$, i.e., $c \gg M^2$, and the trivial bound for the Kloosterman sum, we obtain that
\begin{equation}
 \sum_{c \geq M^2} \frac{S(m,n;c)}{c} H\leg{4 \pi \sqrt{mn}}{c} \ll \frac{SD}{M^{198}},
\end{equation}
which is a satisfactory error term for Lemma \ref{lemma:Kmncalc}.  For the rest of the proof, assume $x > M^{-1}$.  

Our next overall goal is to show that $H(x) = \sum_{\pm} H_{\pm}(x) + O(T^{-400})$, where $H_{\pm}(x)$ are defined by \eqref{eq:Hformula}.  We use this estimate for $x > M^{-1}$, leading to
\begin{equation}
 \sum_{c < M^2} \frac{S(m,n;c)}{c} H\leg{4 \pi \sqrt{mn}}{c} = \sum_{\pm} \sum_{c < M^2} \frac{S(m,n;c)}{c} H_{\pm}\leg{4 \pi \sqrt{mn}}{c} + O(M^3 T^{-400}).
\end{equation}
Recalling $M \leq T^{100}$, this error term is acceptable.  Using the fact that $H_{\pm}$ is small unless \eqref{eq:ctruncation} holds (which we prove below), we may then make this further truncation on $c$ to complete the proof. 

Now we begin the development of $H$ for larger values of $x$ using the integral representation 8.41.11 of \cite{GR} which states
\begin{equation}
 J_{2ir}(x) = \frac{2}{\pi} \int_0^{\infty} \sin(x \cosh(v) - \pi i r) \cos(2r v) dv.
\end{equation}
After some simple manipulations we arrive with the identity
\begin{equation}
\label{eq:JBesselformula}
\frac{J_{2ir}(x) - J_{-2ir}(x)}{\cosh(\pi r)} = \tanh(\pi r) \frac{2}{\pi i} \intR \cos(x \cosh(v)) \e{r v}{\pi} dv.
\end{equation}
We insert \eqref{eq:JBesselformula} into \eqref{eq:Hformula2}.  An integration by parts shows that we can truncate the $v$-integral at $T^{\varepsilon}$ with an error that is $O(x^{-1}(1 + r) \exp(-T^{\varepsilon}))) = O((1+r)\exp(-T^{\varepsilon/2}))$.  Thus we can reverse the orders of integration to get
\begin{equation}
H(x) = \frac{4}{\pi^2} \int_{|v| \leq T^{\varepsilon}} \cos(x \cosh(v)) \int_0^{\infty} r \tanh(\pi r) h_{m,n}(r) \e{rv}{\pi} dr dv + O(T^{-200}).
\end{equation}
Next we insert \eqref{eq:hmnproperty} and $\tanh(\pi r) = 1 + O(e^{-2 \pi r})$, getting
\begin{equation}
H(x) = \frac{8}{\pi^2} \int_{|v| \leq T^{\varepsilon}} \cos(x \cosh v) \int_0^{\infty} r P(r)\exp(-(\frac{r-S}{D})^2) \leg{m}{n}^{ir} \e{rv}{\pi} dr dv + O(T^{-200}).
\end{equation}
Next we change variables $r \rightarrow S + D r$ and extend the $r$-integral to $\mr$ (without a new error term), giving
\begin{multline}
H(x) = D \leg{m}{n}^{iS} 2  \int_{|v| \leq T^{\varepsilon}} \cos(x \cosh v) \e{S v}{\pi} 
\\
\intR \frac{4}{\pi^2} (S + D r) P(S+Dr)\exp(-r^2) \leg{m}{n}^{iDr} \e{D r v}{\pi} dr dv 
+ O(T^{-200}).
\end{multline}
Now we compute the $r$-integral as
\begin{equation}
S \widehat{k}(-\frac{D}{ \pi} (v + \half \log(m/n))), 
\end{equation}
where notice $\widehat{k}$ is a Schwartz-class function satsifying $r^j \widehat{k}^{(j)}(r) \ll_{j,C} (1+|r|)^{-C}$ with implied constants depending on $j$ and $C$ only.  With this definition,
\begin{equation}
H(x) = DS \leg{m}{n}^{iS} 2 \int_{|v| \leq T^{\varepsilon}} \cos(x \cosh v) \e{S v}{\pi} \widehat{k}(-\frac{D}{ \pi} (v + \half \log(m/n))) dv
+ O(T^{-200}).
\end{equation}
Since $|\log(m/n)| \ll R^{-1}$, and $R \ll D T^{\eta}$, if $|v| \gg D^{-1} T^{\eta + \varepsilon}$ then the integrand is very small.  In particular we can extend the range of integration to the whole real line without making a new error term.  
Then we change variables $v \rightarrow v/D - \half \log(m/n)$ and re-truncate the integral, getting
\begin{equation}
H(x) = S  \int_{|v| \leq T^{\varepsilon}} 2\cos(x \cosh (\frac{v}{D} - \half \log(m/n))) \e{S v}{\pi D} \widehat{k}(-\frac{v}{ \pi} ) dv
+ O(T^{-400}).
\end{equation}
Write $2\cos(y) = e^{iy} + e^{-iy}$ and correspondingly write $H(x) = H^0_{+}(x) + H^0_{-}(x)$.  Then
\begin{equation}
H^0_{\pm}(x) = S  \int_{|v| \leq T^{\varepsilon}} \widehat{k}(-\frac{v}{ \pi} ) e^{i \phi(v)} dv + O(T^{-400}),
\end{equation}
where
\begin{equation}
\phi(v) = 2 \frac{S}{D} v \pm x \cosh(\frac{v}{D} - \half \log(m/n)).
\end{equation}
Now we argue that $H^{0}_{\pm}(x)$ is very small if $x \leq \frac{SR}{T^{\varepsilon}}$.  To see this, we write the integral as
\begin{equation}
 \int_{|v| \leq T^{\varepsilon}} f(v) e^{2i\frac{S}{D} v} dv, \qquad f(v) = \widehat{k}(-\frac{v}{ \pi}) e^{\pm i x \cosh(\frac{v}{D} - \half \log(m/n))}.
\end{equation}
A detailed but routine computation shows for $|v| \leq D$ that
\begin{equation}
 f^{(j)}(v) \ll_{j,C} [1 + \frac{x}{D} (\frac{1}{R} + \frac{1}{D})]^j (1 + |v|)^{-C}.
\end{equation}
Thus by Lemma \ref{lemma:basicFourier}, the integral defining $H^{0}_{\pm}$ is very small unless
\begin{equation}
 \frac{S}{D} \ll_{\varepsilon} T^{\varepsilon}[1 + \frac{x}{D}(\frac{1}{D} + \frac{1}{R})]
\end{equation}
Since $S/D \geq T^{\eta}$, by taking $\varepsilon = \eta/2$, say, and recalling $D \gg R T^{-\eta}$, we conclude that 
$H_{\pm}^{0}(x)$ is very small unless $x \gg SR T^{-\varepsilon}$, which is equivalent to \eqref{eq:ctruncation}.

Now we find an alternate formula for $\phi(v)$ to give \eqref{eq:Hformula}.  We begin with the observation
\begin{equation}
 \cosh(\frac{v}{D} - \half \log(m/n)) = \half (\sqrt{\frac{m}{n}} + \sqrt{\frac{n}{m}}) \cosh(v/D) - \half (\sqrt{\frac{m}{n}} - \sqrt{\frac{n}{m}}) \sinh(v/D).
\end{equation}
Since $x = 4 \pi \sqrt{mn}/c$, we have
\begin{equation}
 x \cosh(\frac{v}{D} - \half \log(m/n))= \frac{2\pi (m+n)}{c} \cosh(v/D) - \frac{2\pi (m-n)}{c} \sinh(v/D),
\end{equation}
which simplifies as
\begin{equation}
x \cosh(\frac{v}{D} - \half \log(m/n))= \frac{2 \pi m}{c} \exp(-v/D) + \frac{2 \pi n}{c} \exp(v/D).
\end{equation}
Thus
\begin{equation}
 \phi(v) = 2 \frac{S}{D} v \pm (\frac{2 \pi m}{c} \exp(-v/D) + \frac{2 \pi n}{c} \exp(v/D)).
\end{equation}
We conclude that
\begin{equation}
 H(x) = \sum_{\pm}  S \int_{|v| \leq T^{\varepsilon}} \widehat{k}(-\frac{v}{2 \pi}) e^{2i \frac{S}{D} v} \e{\pm m e^{-v/D}}{c} \e{\pm n e^{v/D}}{c} dv
+ O(T^{-400}).
\end{equation}
This is what we wanted to prove.
\end{proof}

Now we continue with our proof of Theorem \ref{thm:separationofvars}.  We apply Lemma \ref{lemma:Kmncalc} to \eqref{eq:MRQZl55}.  
Write $\mathcal{M}_0$ for the diagonal term contribution, $\mathcal{M}_1$ for the contribution from the sum of Kloosterman sums, and $\mathcal{E}$ for the error terms, i.e., $\mathcal{M} = \mathcal{M}_0 + \mathcal{M}_1 + \mathcal{E}$.  The trivial bound gives
\begin{equation}
 \mathcal{M}_0(R, S, D, Q;a_n) \ll R SD \sum_{n \leq M} |a_m|^2,
\end{equation}
which is satisfactory for Theorem \ref{thm:separationofvars}.  Furthermore,
\begin{equation}
\mathcal{E} \ll \sum_{m \leq M} RSD |a_m|^2
\end{equation}
using Cauchy's inequality, since $M \leq T^{100}$, which is also acceptable for Theorem \ref{thm:separationofvars}.

For the sum of Kloosterman sums, we rewrite $\widehat{g}(\frac{R}{2 \pi} \log\frac{n}{m})$ as an integral, open the Kloosterman sum, and insert absolute values to obtain the following
\begin{multline}
 \mathcal{M}_1(R,S,D,Q, M;a_n) \ll S \sum_{\pm} \intR \int_{|v| \leq T^{\varepsilon}} g(t/R)  |\widehat{k}(-\frac{v}{2\pi})| \sum_{c \leq \frac{M T^{\varepsilon}}{S R}} \frac{1}{c} \sumstar_{r \shortmod{c}} 
\\
\Big|\sum_{m \leq M} \overline{a_m} m^{it} \e{rm}{c} \e{\pm m e^{-v/D}}{c}  \Big|
\Big|\sum_{n \leq M} a_n n^{-it} \e{\overline{r}m}{c} \e{\pm n e^{v/D}}{c}  \Big| dv dt.
\end{multline}
Next we apply the Cauchy-Schwartz inequality and perform some simplifications, in particular writing $|\widehat{k}(-\frac{v}{2 \pi})| \ll 1$ for $|v| \leq T^{\varepsilon}$, to get
\begin{multline}
\label{eq:M1bound55}
 \mathcal{M}_1(R,S,D,Q;a_n) \ll S \sum_{\pm} \intR \int_{|v| \leq T^{\varepsilon}}  g(t/R)  \sum_{c \leq \frac{M T^{\varepsilon}}{S R}} \frac{1}{c} \sumstar_{r \shortmod{c}} 
\\
\Big|\sum_{n \leq M} a_n n^{-it} \e{(r \pm 1)n}{c} \e{\pm n (e^{v/D}-1)}{c}  \Big|^2 dv dt.
\end{multline}

The $v$-integral in \eqref{eq:M1bound55} is set up to apply Lemma \ref{lemma:linear}, effectively replacing $\e{\pm n(e^{v/D} - 1)}{c}$ by $\e{\pm n v}{cD}$ with a very small error term.  That is, after some simple manipulations we have
\begin{multline}
\label{eq:M1bound66}
 \mathcal{M}_1(R,S,D,Q;a_n) \ll S \sum_{\pm} \intR \intR  g(t/R) g(v/T^{2\varepsilon}) \sum_{c \leq \frac{M T^{\varepsilon}}{S R}} \frac{1}{c} \sumstar_{r \shortmod{c}} 
\\
\Big|\sum_{n \leq M} a_n n^{-it} \e{(r \pm 1)n}{c} \e{n v}{Dc}  \Big|^2 dv dt + O(T^{-100} \sum_{n \leq M} |a_n|^2).
\end{multline}
Next we apply Lemma \ref{lemma:conversion} to convert the $n^{-it}$ twist by an additive twist.  In this way we obtain, with $\alpha_n = a_n \e{(r \pm 1)n}{c}$,
\begin{multline}
\label{eq:conversionSection6}
 \intR \intR  g(t/R) g(v/T^{2\varepsilon}) \Big|\sum_{n \leq M} \alpha_n n^{-it} \e{nv}{Dc}\Big|^2 dv dt 
\\
\ll \int_{v} \int_{y \ll R} g(v/T^{2\varepsilon}) \Big|\sum_{n \leq M}  \alpha_n  e(\frac{nv}{cD} + \frac{ny}{M})\Big|^2 dy dv + O(R^{-100} T^{\varepsilon} \sum_{n \leq M} |a_n|^2). 
\end{multline}
Then we change variables $v \rightarrow v - \frac{ycD}{M}$ and replace the ranges of integration by $y \ll R$ and $v \ll T^{\varepsilon} + \frac{RcD}{M} \leq 2T^{\varepsilon}$, using \eqref{eq:ctruncation} and the fact that $D \ll S$.  Thus
the quantity on the right hand side of \eqref{eq:conversionSection6} is
\begin{equation}
 \ll R \int_{|v| \ll T^{\varepsilon}} |\sum_{n \leq M} \alpha_n \e{nv}{cD}|^2 dv 
+
O(R^{-100} T^{\varepsilon} \sum_{n \leq M} |a_n|^2).
\end{equation}
This procedure effectively removes the $t$-integral from the right hand side of \eqref{eq:M1bound66}.

Write $(r \pm 1, c) = a$ and change variables $c = ab$, $r = \mp 1 + au$ where $u$ runs modulo $b$ such that $(u,b) = 1$ and $(au \mp 1, b) = 1$.  By positivity we drop this latter condition. Simplifying completes the proof of Theorem \ref{thm:separationofvars}.

\section{The large sieve}
\label{section:largesieve}
With Theorem \ref{thm:separationofvars} combined with the large sieve (Lemma \ref{lemma:largesieve}) we are able to make significant progress on bounding $\mathcal{M}(R,S,D,Q, M;a_n)$.  We first make a small simplification and set some notation.  For arbitrary complex numbers $b_n$, let 
\begin{equation}
\label{eq:MABdef}
\mathcal{M}_{A, B}(R,S, D, Q, M;b_n) = \frac{R S}{B} \intR g(\frac{v}{T^{\varepsilon}})  \sum_{b \asymp B} \medspace \sumstar_{r \shortmod{b}} \Big|\sum_{n \leq M} b_n  \e{rn}{b} \e{vn }{AB D} \Big|^2 dv.
\end{equation}
%Then with $b_n = a_n n^{-it}$ with $a_n$ given by \eqref{eq:andef}
With this notation, we claim that Theorem \ref{thm:separationofvars} reads
\begin{equation}
\label{eq:thm6.5variant}
\mathcal{M}(R,S, D, Q, M; a_n) \ll R SD  \sum_{n \leq M} |a_n|^2 
+     \sum_{A, B} \mathcal{M}_{A, B}(R,S, D, Q, M;a_n),
\end{equation}
where the sum over $A$ and $B$ is over powers of $2$, say, such that
\begin{equation}
\label{eq:ABtruncation}
AB \ll \frac{M T^{\varepsilon}}{SR}.
\end{equation}
This is immediate after changing variables $v \rightarrow \frac{ab}{AB} v$, extending $v$ by positivity to say $8 T^{\varepsilon}$, and summing trivially over $a \asymp A$.

\begin{mylemma}
For any complex numbers $b_n$, we have 
\begin{equation}
 \mathcal{M}_{A, B}(R,S, D, Q, M;b_n) \ll (RSB + RDSA) T^{\varepsilon} \sum_{n \leq M} |b_n|^2.
\end{equation}
\end{mylemma}
\begin{proof}
 Applying the additive character version of Lemma \ref{lemma:largesieve}, we immediately have
\begin{equation}
 \mathcal{M}_{A, B}(R,S, D, Q, M;b_n) \ll \frac{RS}{AB} A (B^2 T^{\varepsilon} + ABD T^{\varepsilon}) \sum_{n \leq M} |b_n|^2. \qedhere
\end{equation}
\end{proof}
\begin{mycoro}
\label{coro:smallA}
 If $A \leq \frac{N}{RSD} T^{\varepsilon}$ then
\begin{equation}
\mathcal{M}_{A, B}(R,S, D, Q, M;a_n) \ll Q^{\half + \varepsilon} \sum_{n \leq M} |a_n|^2.
\end{equation}
\end{mycoro}
\begin{proof}
 We recall $B \leq \frac{MT^{\varepsilon}}{SR}$, so a short calculation immediately gives the result, recalling $N = Q^{\half + \varepsilon}$.
\end{proof}

For ease of reference, recall that \eqref{eq:MandMl} gives the relation between our main quantity of interest, $\mathcal{M}(R,S,D,Q)$, and $\mathcal{M}(R,S,D,Q,M;a_n)$.
Unravelling the definitions, we have that the contribution to $\mathcal{M}(R,S,D,Q)$ from $A \leq \frac{N}{RSD}  T^{\varepsilon}$ is %say $\mathcal{M}_1(R,S,D,Q)$ satisfies
\begin{equation}
%\mathcal{M}_1(R,S,D,Q) 
\ll T^{\varepsilon} \sup_{1 \ll M \ll N} \sum_{l \leq \sqrt{N/M}} l^{-1} (RSD +N) \sum_{n \leq M} \frac{|A_F(l,n)|^2}{n} \ll Q^{\half + \varepsilon} \sum_{l^2 n \leq N} \frac{|A_F(l,n)|^2}{ln}.
\end{equation}
Then recall the statement of Lemma \ref{lemma:Molteni}.

It is perhaps surprising how much progress one makes without using any special properties of the coefficients $a_n$.  Since the variable $a$ occurs as the greatest common divisor of two integers one might expect that $a=1$ is the most important case, but unfortunately larger values of $a$ are problematic and require new ideas.
For the complementary ranges of $A$, i.e. $A \geq \frac{N}{RSD} T^{-\varepsilon}$ we resorted to using the $GL_{3}$ Voronoi formula.  We will see that for such sizes of $A$ then the Voronoi formula is beneficial in the sense that the dual sum is shorter than the original sum.

\section{Applying the $GL_{3}$ Voronoi formula}
\label{section:Voronoi}
In this section we shall apply the $GL_{3}$ Voronoi formula to obtain some crucial additional savings when $A$ is relatively large.  
% The main result of this section is the following
% \begin{mytheo}
%  Let $\mathcal{M}_{A,B}$ be given by \eqref{eq:MABdef}, where
% $a_n$ is specialized to be \eqref{eq:andef}, and $AB \ll \frac{MT^{\varepsilon}}{QR}$.  Then
% \begin{equation}
%  \mathcal{M}_{A,B}(R,Q;Z,a_n) \ll ??.
% \end{equation}
% \end{mytheo}
% [Presumably now we can show how this implies a bound on $\mathcal{M}$ which is satisfactory for $A \gg ??$.]
We begin by fixing some new notation.
We write \eqref{eq:MABdef} as
\begin{equation}
\label{eq:MABwrtS}
 \mathcal{M}_{A,B}(R, S, D, Q, M; a_n) = \frac{RS}{B}  \sum_{b \asymp B} \medspace \sumstar_{r \shortmod{b}} \mathcal{S}(b,r,v),
\end{equation}
where for brevity we have not displayed all the variables of $\mathcal{S}$, and
\begin{equation}
 \mathcal{S}(b,r,v) = \intR g(v/T^{\varepsilon}) M^{-1} \Big|\sum_{n \geq 1} A_F(l,n) \e{rn}{b} \eta(n) n^{-iT_0} \e{vn}{ABD} \Big|^2 dv, 
\end{equation}
where
\begin{equation}
 \eta(n) = \leg{n}{M}^{-\half} W(nl^2) \gamma(n/M).
\end{equation}
Notice that $\eta$ satisfies
\begin{equation}
\label{eq:etabound}
 x^j \eta^{(j)}(x) \ll_{j,C} (1 + \frac{x}{M})^{-C}.
\end{equation}
Now we apply Theorem \ref{thm:Voronoi} and Cauchy's inequality to $\mathcal{S}$.  We obtain
\begin{equation}
\label{eq:SafterVoronoi}
 \mathcal{S} \ll \sum_{\pm} \sum_{k \in \{0, 1\}} \intR g(v/T^{\varepsilon}) M^{-1} \Big|b \sum_{d | bl} \sum_{n \geq 1} \frac{A_F(n,d)}{dn} S(\pm \overline{r} l, n; \frac{bl}{d}) \Phi_{k}(\frac{nd^2}{b^3 l}, v)\Big|^2 dv.
\end{equation}
Here
\begin{equation}
\label{eq:Phixvdef}
 \Phi_{k}(x,v) = \frac{1}{2 \pi i} \int_{(\sigma)}  (\pi^2 x)^{-s} 
\frac{\Gamma(\frac{1+ \sigma + it + i\alpha+k}{2})}{\Gamma(\frac{- \sigma - it - i\alpha+k}{2})} 
\frac{\Gamma(\frac{1+ \sigma + it + i \beta+k}{2})}{\Gamma(\frac{- \sigma - it-i\beta+k}{2})} 
\frac{\Gamma(\frac{1+ \sigma + it + i\gamma+k}{2})}{\Gamma(\frac{- \sigma - it - i\gamma+k}{2})}
\widetilde{\phi}_{T_0}(-s)
 ds,
% \end{equation}
% \begin{equation}
\end{equation}
where $s = \sigma + it$, and
\begin{equation}
\label{eq:phitildedef}
 \widetilde{\phi}_{T_0}(-\sigma - it) = \int_0^{\infty} \eta(x) x^{-iT_0} e^{\frac{2 \pi i vx}{ABD}} x^{-\sigma - it} \frac{dx}{x}.
\end{equation}
The cases $k=0$ and $k=1$ are very similar.  
Set  
\begin{equation}
\label{eq:Udef}
 U = \frac{M}{ABD}.
\end{equation}

\begin{mylemma}
\label{lemma:Phiproperties}
Let $V = \alpha-\beta$ if $T_0 = \alpha$ or $T_0 = \beta$, and $V= T$ if $T_0 = \gamma$.  Then
we have the bound for sufficiently large $\sigma > 0$
\begin{equation}
\label{eq:Phitruncation}
\Phi_k(x,v) \ll_{\sigma, \varepsilon} \leg{U(U + T)(U+V) T^{\varepsilon}}{xM}^{\sigma}.
\end{equation}
Furthermore, suppose that $b_m$ is an arbitrary finite sequence of complex numbers, and that $g$ is a nonnegative smooth function with compactly-supported Fourier transform.  Then with $\widetilde{\phi}$ given by \eqref{eq:phitildedef}, and any real $c > 0$, we have
\begin{multline}
\label{eq:Phibilinear}
\intR g\leg{v}{T^{\varepsilon}} | \sum_{m \geq 1} b_m \Phi_k(\frac{m}{c}, v)|^2 dv \ll  \frac{M T^{\varepsilon}}{U} \int_{|t| \leq T^{\varepsilon} U} |\sum_{m \geq 1} b_m m^{iT_0} \sqrt{\frac{m}{c}} m^{it}|^2 dt 
\\
+ T^{-100} \sum_{m \geq 1} |b_m|^2.
\end{multline}
\end{mylemma}
The pleasant feature of this Lemma is that we avoided a difficult asymptotic analysis of the complicated function $\Phi_k$.  The method of proof can be applied in many other situations.

\begin{proof}
We first prove \eqref{eq:Phitruncation}.  Choose $\sigma > 0$ very large compared to $\varepsilon$, and change variables $s \rightarrow s - iT_0$ in the definition \eqref{eq:Phixvdef}.  Notice that $\widetilde{\phi}_{T_0}(-\sigma - i(t-T_0)) =:\widetilde{\phi}(-\sigma - it)$ does not depend on $T_0$; indeed,
\begin{equation}
 \widetilde{\phi}(-\sigma - it) =  \int_0^{\infty} \eta(x) e^{\frac{2 \pi i vx}{ABD}} x^{-\sigma - it} \frac{dx}{x}.
\end{equation}
First note the very crude bound $\widetilde{\phi}(-\sigma - it) \ll M^{-\sigma}$ and Stirling's approximation
\begin{multline}
\left| \frac{\Gamma(\frac{1+ \sigma + it + i(\alpha-T_0)+k}{2})}{\Gamma(\frac{- \sigma - it - i(\alpha-T_0)+k}{2})} 
\frac{\Gamma(\frac{1+ \sigma + it + i (\beta-T_0)+k}{2})}{\Gamma(\frac{- \sigma - it-i(\beta-T_0) +k}{2})} 
\frac{\Gamma(\frac{1+ \sigma + it + i(\gamma-T_0)+k}{2})}{\Gamma(\frac{- \sigma - it - i(\gamma-T_0)+k}{2})}
 \right| 
\\
\ll (1 + |t+(\alpha-T_0)|)^{\half + \sigma} (1 + |t+(\beta-T_0)|)^{\half + \sigma} (1 + |t+(\gamma-T_0)|)^{\half + \sigma}.
\end{multline}
Next, we note that if $|t| \geq U T^{\varepsilon}$ then integration by parts shows that $\widetilde{\phi}(-\sigma - it) \ll_C M^{-\sigma} |t|^{-C}$ for $C > 0$ arbitrarily large.  Since $\alpha - \gamma \asymp \beta - \gamma \asymp T$, we have for all three choices of $T_0$ that
\begin{equation}
\Phi_k(x,v) \ll (xM)^{-\sigma} \intR \frac{(1 + |t|)^{\half + \sigma} (1 + |t| + V)^{\half + \sigma} (1 + |t| + T)^{\half + \sigma}}{(1 + \frac{|t|}{U T^{\varepsilon}})^C} dt,
\end{equation}
which directly gives
\begin{equation}
\Phi_k(x,v) \ll T^{\varepsilon} U^{3/2}(U + T)^{1/2} (U + V)^{1/2} \leg{U(U+T)(U + V)}{xM}^{\sigma}.
\end{equation}
Choosing $\sigma$ large enough compared to $\varepsilon$ gives \eqref{eq:Phitruncation}.

Next we prove \eqref{eq:Phibilinear}.  One could attempt to prove this by finding an asymptotic expansion of $\widetilde{\phi}$, then applying the asymptotic form of Stirling's approximation, opening the square and analyzing the triple integral with methods of oscillatory integrals.  This is feasible, but it is very complicated, so it is extremely nice that there is a simpler method presented below.  It is reminiscent of 
the calculation of the magnitude of a Gauss sum by computing its modulus squared; of course, the magnitude is much easier to calculate than the argument.  The proof follows the same lines as in Lemma \ref{lemma:conversion}.

For the rest of the proof we fix $\sigma = -\half$.  Define
\begin{equation}
G(t) = \frac{1}{2 \pi} \pi^{-\frac32 -3 \sigma - 3it} 
\frac{\Gamma(\frac{1+ \sigma + it + i(\alpha-T_0) +k}{2})}{\Gamma(\frac{- \sigma - it - i(\alpha-T_0) + k}{2})} 
\frac{\Gamma(\frac{1+ \sigma + it + i(\beta-T_0) +k}{2})}{\Gamma(\frac{- \sigma - it - i(\beta-T_0) + k}{2})} 
\frac{\Gamma(\frac{1+ \sigma + it + i(\gamma-T_0) +k}{2})}{\Gamma(\frac{- \sigma - it - i(\gamma-T_0) + k}{2})} .
\end{equation}
Similarly, let
\begin{equation}
H(u) = \leg{u}{M}^{-\sigma} \eta(u).
\end{equation}
Then with these notations we have
\begin{equation}
\Phi_k(x,v) = x^{-\sigma+iT_0} \int_{|t| \leq UT^{\varepsilon}} \widetilde{\phi}(-\sigma - it) x^{-it} G(t) dt + O(T^{-200}),
\end{equation}
the $t$-truncation coming from the rapid decay of $\widetilde{\phi}$, and
\begin{equation}
\widetilde{\phi}(-\sigma - it) = M^{-\sigma} \int_0^{\infty} H(u) u^{-it} \e{vu}{ABD} \frac{du}{u}.
\end{equation}
Note that $H$ satisfies \eqref{eq:etabound}.  Let $J$ be the left hand side of \eqref{eq:Phibilinear}, and write $J =J_1 + (\text{error})$, where this acceptable error comes from the $t$-truncation.  Then
\begin{multline}
J_1= \sum_{m,n \geq 1} b_m \overline{b_n} \leg{m}{c}^{-\sigma + iT_0} \leg{n}{c}^{-\sigma - iT_0} M^{-2\sigma} \int_{-T^{\varepsilon} U}^{T^{\varepsilon} U} \int_{-T^{\varepsilon} U}^{T^{\varepsilon} U} \leg{m}{c}^{-i t_1} \leg{n}{c}^{it_2} G(t_1) \overline{G(t_2)} 
\\
\int_0^{\infty} \int_{0}^{\infty} H(u_1) \overline{H(u_2)} u_1^{-it_1} u_2^{it_2} 
\intR g(\frac{v}{T^{\varepsilon}}) \e{v(u_1 - u_2)}{ABD} dv \frac{du_1 du_2}{u_1 u_2} dt_1 dt_2.
\end{multline}
This inner $v$-integral is $T^{\varepsilon} \widehat{g}(\frac{T^{\varepsilon}(u_2-u_1)}{ ABD})$, which is zero unless $|u_1 - u_2| \ll ABD T^{-\varepsilon} \asymp \frac{M}{U T^{\varepsilon}}$, recalling \eqref{eq:Udef}. For reference, $u_1, u_2$ are of size $M$ by the support of $H$.  Having imposed this condition, we move the $v$, $u_1$, and $u_2$-integrals to the outside, getting
\begin{multline}
J_1= \intR g(\frac{v}{T^{\varepsilon}}) \mathop{\int \int}_{|u_1 - u_2| \ll \frac{M}{UT^{\varepsilon}}} H(u_1) \overline{H(u_2)} \e{v(u_1 - u_2)}{ABD}
M^{-2\sigma}  
\sum_{m,n \geq 1} b_m \overline{b_n} \leg{m}{c}^{-\sigma + iT_0} 
\\
\leg{n}{c}^{-\sigma - iT_0} 
\int_{-T^{\varepsilon} U}^{T^{\varepsilon} U} \int_{-T^{\varepsilon} U}^{T^{\varepsilon} U} \leg{m}{c}^{-i t_1} \leg{n}{c}^{it_2} G(t_1) \overline{G(t_2)} 
 u_1^{-it_1} u_2^{it_2} 
  dt_1 dt_2 dv \frac{du_1 du_2}{u_1 u_2}.
\end{multline}
We write this in the form
\begin{equation}
|J_1| \leq \int_v \int_{u_1} \int_{u_2}  \Big|\sum_m \int_{t_1}  \Big| \Big|\sum_n  \int_{t_2} \Big|,
\end{equation}
and apply the inequality $|X| |Y| \leq \half (|X|^2 + |Y|^2)$.  Both terms lead to the same expression by symmetry.  Integrating trivially over $v$ and the $u_i$ not occuring inside the square, we then obtain
\begin{equation}
 |J_1| \ll \frac{M}{U} M^{-2\sigma} \int_u |H(u)|^2 \Big|\sum_{m \geq 1} b_m \leg{m}{c}^{-\sigma + iT_0} \int_{-T^{\varepsilon} U}^{T^{\varepsilon} U} \leg{m}{c}^{-it} G(t) u^{-it} dt \Big|^2 \frac{du}{u^2}. 
\end{equation}

We now have one fewer integral sign inside the square, compared to the original definition.  Our next step is to do the same procedure to eliminate the $t$-integral on the inside.  Opening up the square again, we have
\begin{multline}
\label{eq:J1lastbound}
|J_1| \ll \frac{M}{U}  M^{-2\sigma}  \sum_{m,n \geq 1} b_m \overline{b_n} \leg{m}{c}^{-\sigma + iT_0} \leg{n}{c}^{-\sigma - iT_0} \int_{t_1} \int_{t_2} \leg{m}{c}^{-it_1} \leg{n}{c}^{it_2} G(t_1) \overline{G(t_2)} 
\\
\intR |H(u)|^2 u^{-it_1 + it_2-1} \frac{du}{u} dt_1 dt_2.
\end{multline}
Integration by parts shows that the inner $u$-integral is very small unless $|t_1 - t_2| \leq T^{\varepsilon}$.  According to this truncation, write the right hand side of \eqref{eq:J1lastbound} as $J_2 + (\text{error})$, where the error is acceptable for the proof.  Having imposed this condition, move the $u$, $t_1$, and $t_2$-integrals to the outside and put in absolute value signs as follows:
\begin{multline}
 J_2 \leq \frac{M}{U} M^{-2\sigma} \intR |H(u)|^2 u^{-1} \mathop{\int \int}_{|t_1 - t_2| \leq T^{\varepsilon}} 
\\
\Big|\sum_{m \geq 1} b_m \leg{m}{c}^{-\sigma + iT_0 - it_1} G(t_1)\Big| \Big|
\sum_{n \geq 1}  b_n  \leg{n}{c}^{-\sigma + iT_0-it_2}  G(t_2)\Big| dt_1 dt_2 \frac{du}{u}. 
\end{multline}
As in the above treatment of $J_1$, we use Cauchy-Schwartz on the triple integral, giving
\begin{equation}
 J_2 \ll \frac{M}{U} M^{-2\sigma} \intR |H(u)|^2 u^{-1} \mathop{\int \int}_{|t_1 - t_2| \leq T^{\varepsilon}} \Big|\sum_{m \geq 1} b_m \leg{m}{c}^{-\sigma + iT_0 - it_1} G(t_1)\Big|^2 dt_1 dt_2 \frac{du}{u}.
\end{equation}
We bound the $u$- and $t_2$-integrals trivially, getting
\begin{equation}
 J_2 \ll \frac{T^{\varepsilon}}{U} M^{-2\sigma} \int_{|t| \leq T^{\varepsilon} U} |G(t)|^2  \Big|\sum_{m \geq 1} b_m \leg{m}{c}^{-\sigma + iT_0 - it} \Big|^2 dt.
\end{equation}
Note the wonderful fact that $|G(y)|^2 = 1$ for $\sigma = -\half$! Thus we get
\begin{equation}
J_2 \ll  \frac{M T^{\varepsilon}}{U} \int_{|t| \leq T^{\varepsilon}U} \Big| \sum_{m \geq 1} b_m m^{iT_0} \sqrt\frac{m}{c} m^{it} \Big|^2 dt.
\end{equation}
This is what we wanted to prove.
\end{proof}

\section{Reduction to the large sieve}
\label{section:largesieve2}
In view of \eqref{eq:MandMl} and \eqref{eq:thm6.5variant}, write
\begin{equation}
 \mathcal{P}_{A,B}(R,S,D,M) = \sum_{l \leq \sqrt{N}} \frac{1}{l} M_{A,B}(R,S, D, Q, M;a_n),
\end{equation}
where recall the definition \eqref{eq:MABdef}.
\begin{mylemma}
\label{lemma:bigA}
 Suppose $A \geq \frac{N}{RSD} T^{-\varepsilon}$ where \eqref{eq:RDST} holds.  Then
\begin{equation}
 \mathcal{P}_{A,B}(R, S, D, M) \ll Q^{\half + \varepsilon} |A_F(1,1)|^2.
\end{equation}
\end{mylemma}
Combining Lemma \ref{lemma:bigA} with Corollary \ref{coro:smallA}, we complete the proof of Theorem \ref{thm:mainthm}.

We state and prove some elementary results used in the proof of Lemma \ref{lemma:bigA}.
\begin{mylemma}
\label{lemma:9.2}
Let $c_m$ be an arbitrary finite sequence of complex numbers, and suppose $r|b^{\infty}$, meaning all the prime factors dividing $r$ also divide $b$.  Then
\begin{equation}
 \sum_{x \shortmod{b}} \Big| \sum_{m \geq 1} c_m S(rx, m ;br) \Big|^2 =
b r^2 \sumstar_{y \shortmod{b}} \Big| \sum_{m \equiv 0 \shortmod{r}} c_{m} \e{y \frac{m}{r}}{b} \Big|^2
\end{equation}
\end{mylemma}
\begin{proof}
 Opening the square, writing out the definition of the Kloosterman sum, and evaluating the sum over $x$ using orthogonality of characters, we have
\begin{equation}
\label{eq:13.4}
\sum_{x \shortmod{b}} \Big| \sum_m c_m S(rx, m ;br) \Big|^2 = b \sum_{m_1, m_2} c_{m_1} \overline{c_{m_2}} \sumstar_{\substack{h_1, h_2 \shortmod{br} \\h_1 \equiv h_2 \shortmod{b}}} \e{h_1 m_1 - h_2 m_2}{br}.
\end{equation}
Change variables via $h_i = y + b z_i$, $i=1,2$, where $y$ runs modulo $b$ and $z_i$ runs modulo $r$.  Since $r |b^{\infty} $, the condition that $(h_i, br) = 1$ is equivalent to $(y, b) = 1$.  The sum over $z_i$ vanishes unless $r | m_i$, in which case the sum is $r$. Thus \eqref{eq:13.4} equals
\begin{equation}
b r^2 \sum_{r|m_1, m_2} c_{m_1} \overline{c_{m_2}} \sumstar_{y \shortmod{b}} \e{y (\frac{m_1}{r} - \frac{m_2}{r}}{b},
\end{equation}
which is easily rewritten to complete the proof.
\end{proof}

\begin{mylemma}
\label{lemma:9.3}
 Let $b_m$ be an arbitrary finite sequence of complex numbers.  Then
 \begin{equation}
\label{eq:11.7}
\Big| \sum_{m \geq 1} b_m S(0,m;s) \Big|^2 \leq s \sumstar_{h \shortmod{s}} \Big| \sum_{m \geq 1} b_m \e{hm}{s} \Big|^2.
\end{equation}
\end{mylemma}
\begin{proof}
 This follows from opening the Kloosterman sum, reversing the orders of summation, and applying Cauchy's inequality to the outer sum.
\end{proof}

\begin{mylemma}
\label{lemma:9.4}
 Suppose $(b,s) = 1$, $r | b^{\infty}$, and $a_m$ is an arbitrary finite sequence of complex numbers.  Then
\begin{equation}
\label{eq:9.4}
 \sumstar_{x \shortmod{b}} \Big|\sum_{m \geq 1} a_m S(0, m;s) S(rx,m;br)\Big|^2 \leq br^2 s \sumstar_{x \shortmod{bs}} \Big|\sum_{m \equiv 0 \shortmod{r}} a_m \e{x \frac{m}{r}}{bs}\Big|^2.
\end{equation}
\end{mylemma}
\begin{proof}
Let $S$ be the left hand side of \eqref{eq:9.4}. Letting $c_m = a_m S(0,m;s)$ and applying Lemma \ref{lemma:9.2}, we have that
\begin{equation}
 S \leq br^2 \sumstar_{y \shortmod{b}} \Big|\sum_{m \equiv 0 \shortmod{r}} a_m S(0,m ;s) \e{y\frac{m}{r}}{b} \Big|^2.
\end{equation}
Next we apply Lemma \ref{lemma:9.3} with $b_m = a_m \e{y \frac{m}{r}}{b}$, getting
\begin{equation}
 S \leq br^2 s \sumstar_{h \shortmod{s}} \sumstar_{y \shortmod{b}} \Big| \sum_{m \geq 1} a_m \e{hr \frac{m}{r}}{s} \e{y \frac{m}{r}}{b}\Big|^2.
\end{equation}
Finally we change variables $h \rightarrow \overline{r} h$, valid since $(r,s) = 1$, and write $x = hb + ys$ which by the Chinese remainder theorem runs over $(\mz/bs \mz)^*$.
\end{proof}

\begin{proof}[Proof of Lemma \ref{lemma:bigA}]
Recall \eqref{eq:MABwrtS} and  \eqref{eq:SafterVoronoi}. %, and Lemma \ref{lemma:Phiproperties}.  
In this way we get (we do not display all of the parameters of $\mathcal{P}_{A,B}$)
\begin{multline}
 \mathcal{P}_{A,B} \ll \frac{RS}{B} \sum_{l \leq \sqrt{N}} \frac{1}{l} \sum_{b \asymp B} \medspace \sumstar_{x \shortmod{b}} 
\\
\intR g(v/T^{\varepsilon}) \sum_{\pm, k} M^{-1} \Big|b \sum_{y | bl} \sum_{n \geq 1} \frac{A_F(n,y)}{yn} S(\pm x l, n; \frac{bl}{y}) \Phi_k(\frac{ny^2}{b^3 l}, v)\Big|^2 dv.
\end{multline}
Applying Cauchy's inequality to the sum over $y$, we obtain
\begin{multline}
 \mathcal{P}_{A,B} \ll \frac{RS T^{\varepsilon}}{B} \sum_{l \leq \sqrt{N}} \frac{1}{l} \sum_{b \asymp B} \medspace \sumstar_{x \shortmod{b}} \sum_{y | bl}
\\
\sum_{\pm, k} \intR g(v/T^{\varepsilon})   M^{-1} \Big|b  \sum_{n \geq 1} \frac{A_F(n,y)}{yn} S(\pm x l, n; \frac{bl}{y}) \Phi_k(\frac{ny^2}{b^3 l}, v)\Big|^2 dv.
\end{multline}
We apply Lemma \ref{lemma:Phiproperties}, truncating the sum over $n$ with \eqref{eq:Phitruncation}.
In this way we obtain, noting that the choice of $\pm$ sign and choice of $k$ lead to the same upper bound, 
\begin{multline}
 \mathcal{P}_{A,B} \ll \frac{RS T^{\varepsilon}}{BU} \sum_{l \leq \sqrt{N}} \frac{1}{l} \sum_{b \asymp B} \medspace \sumstar_{x \shortmod{b}} \sum_{y | bl}
\int_{-UT^{\varepsilon}}^{UT^{\varepsilon}} \Big| \sum_{n \leq N_2^*} \frac{A_F(n,y)'}{\sqrt{nbl}} S(x l, n; \frac{bl}{y}) n^{it}\Big|^2 dt 
\\
+ O(T^{-50}), %\sum_{q_1 q_2 \ll T^{50}} \frac{|A_F(q_1, q_2)|^2}{q_1 q_2},
\end{multline}
where $A_F(n,y)' = A_F(n,y) n^{iT_0}$ and 
\begin{equation}
\label{eq:N2size}
 M \frac{N_2^* y^2}{B^3 l} \ll T^{\varepsilon} U(T+U)(V+U).
\end{equation}
For simplicity, we restrict the variables to dyadic segments as follows: $l \asymp L$, $y \asymp Y$, $n \asymp N_2$, writing $\mathcal{P}_{A,B} \ll \sum_{L, Y, N_2} \mathcal{P}_{A,B}(L,Y,N_2) + O(T^{-50})$, where $L$, $Y$, $N_2$ run over dyadic numbers, 
% and
% The truncation on $n$ is such that
% \begin{equation}
% \label{eq:N2size}
%  M \frac{N_2 Y^2}{B^3 L} \ll T^{\varepsilon} U(T+ U)(V +U).
% \end{equation}
Rearranging \eqref{eq:N2size} and recalling \eqref{eq:Udef}, we have
\begin{equation}
\label{eq:N2*def}
 N_2 \ll \frac{B^2 L}{AD Y^2} (T+ \frac{M}{ABD})(V +\frac{M}{ABD}) T^{\varepsilon}.
\end{equation}
We recall that in our application, $ML^2 \leq N =Q^{\half + \varepsilon}$,  $AB \leq \frac{M}{RS} T^{\varepsilon}$, and $A \geq \frac{N}{RSD} T^{-\varepsilon}$.

The reader who considers only the case $l=y =1$ can finish the proof fairly easily using the large sieve.  Unfortunately, there are other cases that require a more involved treatment; in particular, in the opposite extreme case with $y=b$ then the Kloosterman sum above has modulus $l$ and one observes that the sums over $x$ and $b$ must be executed trivially.  In this case we need to exploit $l$ as a modulus.  In general we need to ``interpolate'' between these two extreme cases ($y=1$ and $y=b$) and partially combine $b$ and $l$ into one modulus.  This is the underlying motivation behind the forthcoming arguments.

Now we do some elementary arrangements.  Write $(b,y) = d$ and change variables $b \rightarrow db$, $y \rightarrow dy$, getting
\begin{multline}
 \mathcal{P}_{A,B}(L,Y,N_2) \ll \frac{RS T^{\varepsilon}}{B U L} \sum_{l \asymp L} \sum_{d \ll \min(B,Y)} \sum_{b \asymp \frac{B}{d}} \medspace \sumstar_{x \shortmod{bd}} 
\\
 \int_{-UT^{\varepsilon}}^{UT^{\varepsilon}}  \sum_{\substack{y | l, y\asymp Y/d \\ (b,y) = 1}} \Big| \sum_{n \asymp N_2} \frac{A_F(n,dy)'}{\sqrt{nbdl}} S( x l, n; \frac{bl}{y}) n^{it} \Big|^2 dt.
\end{multline}
Then write $l = yrs$ where $r | b^{\infty}$ (meaning all the prime factors of $r$ divide $b$) and $(s,b) = 1$.  This gives
\begin{multline}
\mathcal{P}_{A,B}(L,Y,N_2) \ll 
\frac{RS T^{\varepsilon}}{B U L}  \sum_{d \ll \min(B,Y)} \sum_{b \asymp \frac{B}{d}} \medspace \sumstar_{x \shortmod{bd}} 
 \\
 \int_{-UT^{\varepsilon}}^{UT^{\varepsilon}}
 \sum_{\substack{y\asymp Y/d \\ (b,sy) = 1}} \sum_{\substack{yrs \asymp L \\ r|b^{\infty}}}  \Big| \sum_{n \asymp N_2} \frac{A_F(n,dy)'}{\sqrt{nbdyrs}} S(x yrs, n;brs) n^{it}\Big|^2 dt.
\end{multline}
Although $x$ runs modulo $bd$, the Kloosterman sum is unchanged when replacing $x$ by a multiple of $b$.  The same sum is repeated at most $d$ times, whence
\begin{multline}
\mathcal{P}_{A,B}(L,Y,N_2) \ll \frac{RS T^{\varepsilon}}{B U L}   \sum_{d \ll \min(B,Y)} d \sum_{b \asymp B/d} 
\\
\int_{-UT^{\varepsilon}}^{UT^{\varepsilon}} \sum_{\substack{y\asymp Y/d \\ (b,sy) = 1}} \sum_{\substack{yrs \asymp L \\ r|b^{\infty}}} \frac{1}{brs} \sumstar_{x \shortmod{b}} \Big| \sum_{n \asymp N_2} \frac{A_F(n,dy)'}{\sqrt{ndy}} S(x yrs, n;brs) n^{it}\Big|^2 dt.
\end{multline}
From the multiplicativity relation for Kloosterman sums, we have
\begin{equation}
S(xyrs, n ; brs) = S(y r s \overline{s} x, n \overline{s}; br) S(y rs \overline{br} x, n \overline{br} ;s) = S(y r x \overline{s}, n , br) S(0, n ;s) ,
\end{equation}
which becomes $S(r x, n ;br) S(0, n;s)$ after the change of variables $x \rightarrow s \overline{y} x$ (observe that $y$ is coprime to $br$).
Applying Lemma \ref{lemma:9.4}, we have
\begin{multline}
\mathcal{P}_{A,B}(L,Y,N_2) \ll \frac{RS T^{\varepsilon}}{B U L}   \sum_{d \ll \min(B,Y)} d \sum_{b \asymp B/d} \int_{-UT^{\varepsilon}}^{UT^{\varepsilon}}
\\
 \sum_{\substack{y\asymp Y/d \\ (b,sy) = 1}} \sum_{\substack{yrs \asymp L \\ r|b^{\infty}}} r \sumstar_{h \shortmod{bs}} \Big| \sum_{\substack{n \equiv 0 \shortmod{r} \\ n \asymp N_2}} \frac{A_F(n,dy)'}{\sqrt{ndy}} \e{h \frac{n}{r}}{bs} n^{it}\Big|^2 dt.
\end{multline}
Next say $s \asymp H$ where $H Y r \asymp dL$, (note $H \ll L$) and accordingly write $\mathcal{P}_{A,B}(L,Y, N_2) \ll \sum_{H} \mathcal{P}_{A,B}(L,Y, N_2, H)$. In addition, group $bs = c$ as a new variable and drop the condition $r|b^{\infty}$ by positivity.  We get the new bound
\begin{multline}
\mathcal{P}_{A,B}(L,Y, N_2, H) \ll \frac{RS T^{\varepsilon}}{AB U L}  \sum_{a \asymp A} \sum_{d \ll \min(B,Y)} d  \sum_{\substack{y\asymp Y/d}} \sum_{\substack{r \ll \frac{Ld}{YH} }} r
\\
\int_{-UT^{\varepsilon}}^{UT^{\varepsilon}}
\sum_{c \asymp \frac{BH}{d}}
 \sumstar_{h \shortmod{c}} \Big| \sum_{n \asymp N_2/r} \frac{A_F(nr,dy)'}{\sqrt{nrdy}} \e{h n}{c} n^{it} \Big|^2 dt.
\end{multline}
We next apply the large sieve, Lemma \ref{lemma:largesieve}, getting
\begin{multline}
 \mathcal{P}_{A,B}(L,Y, N_2, H) \ll \frac{RS T^{\varepsilon}}{B U L} 
\sum_{d \ll \min(B,Y)} d  
\\
\sum_{\substack{y\asymp \frac{Y}{d}}} \sum_{\substack{r \ll \frac{Ld}{YH} }} r (U \leg{BH}{d}^2 + \frac{N_2}{r}) dy \sum_{n \asymp N_2/r} \frac{|A_F(nr, dy)|^2}{nr(dy)^2}.
\end{multline}
Making $nr = q_1$ and $dy = q_2$ be new variables and summing appropriately, truncating the innermost sum at, say $q_1 q_2^2 \leq T^{100}$, we have
\begin{equation}
 \mathcal{P}_{A,B}(L,Y, N_2, H) \ll \frac{RS T^{\varepsilon}}{B U L}       ( L U B^2 H + N_2 Y^2) \sum_{q_1 q_2^2 \leq T^{100}} \frac{|A_F(q_1, q_2)|^2}{q_1 q_2^2}.
\end{equation}
By Lemma \ref{lemma:Molteni}, this inner sum is $O(|A_F(1,1)|^2 T^{\varepsilon})$.  Then a small calculation gives, recalling $H \ll L$
\begin{equation}
\label{eq:PABrandom}
 \mathcal{P}_{A,B}(L,Y, N_2, H) \ll |A_F(1,1)|^2 T^{\varepsilon}(RS BL + \frac{RSY^2}{BUL} N_2)
\end{equation}
We observe that the first term inside the parentheses is satisfactory noting that
\begin{equation}
 RSBL \ll RSL \frac{M}{RS} T^{\varepsilon} \asymp LM T^{\varepsilon} \ll \frac{N}{L} T^{\varepsilon}.
\end{equation}
With \eqref{eq:N2size}, we calculate the second term inside the parentheses in \eqref{eq:PABrandom} as
\begin{equation}
 \frac{RSY^2}{BUL} N_2 \asymp \frac{RSB^2}{MU} \frac{N_2 Y^2 M}{B^3 L} \ll \frac{RS B^2}{M} (T+U)(V+U).
\end{equation}
Since $V \ll T$ (recall the definition of $V$ given in Lemma \ref{lemma:Phiproperties}) and $Y \gg 1$, we have that this term is
\begin{equation}
 \ll \frac{RS B^2}{M} (TV + UT + U^2) := I + II + III.
\end{equation}
We calculate each of these terms in turn.  Recall $B \leq \frac{M}{ARS} T^{\varepsilon}$ (see \eqref{eq:ABtruncation}), $A \geq \frac{N}{RSD} T^{-\varepsilon}$, and $M \leq N$, so that
\begin{equation}
 I \ll \frac{MTV}{RSA^2} T^{\varepsilon} \ll \frac{TV D^2 RS}{N} T^{\varepsilon}.
\end{equation}
We recall that $N = Q^{\half + \varepsilon}$ and $Q \asymp T^2 DR(S + (\alpha-\beta))(1 + (\alpha-\beta))$.  Thus
\begin{equation}
 TVD^2 RS = TDR (VSD) \ll TDR (T(S + \alpha-\beta)(1 + \alpha-\beta)) \asymp Q,
\end{equation}
since $V \ll \alpha-\beta$ unless $T_0 = \gamma$ in which case $V= T$ and $D \ll \alpha - \beta$, recalling Lemma \ref{lemma:meanvaluereduction}.  Thus $I \ll N T^{\varepsilon}$, as desired.

We calculate
\begin{equation}
 II \ll \frac{RSBT}{AD} \ll \frac{T M}{A^2 D} T^{\varepsilon} \ll \frac{T R^2 DS^2}{N} T^{\varepsilon}.
\end{equation}
We claim $T R^2 DS^2 \ll Q$, which follows from
\begin{equation}
 RS^2 \ll (1 + \alpha-\beta) (S + \alpha-\beta) T.
\end{equation}
Thus $II \ll N T^{\varepsilon}$, as desired.

Finally, we have
\begin{equation}
 III \asymp \frac{RS M}{A^2 D^2} \ll \frac{R^3 S^3 M}{N^2} T^{\varepsilon} \ll \frac{R^3 S^3}{N} T^{\varepsilon}.
\end{equation}
Then we check $R^3 S^3 \ll Q$, whence $III \ll N T^{\varepsilon}$.
\end{proof}

\end{document}